\documentclass[12pt]{amsart}
\usepackage{amssymb}
\usepackage{mathtools}
\usepackage[english]{babel}
\usepackage{epsfig}
\usepackage{mathptmx}
\usepackage{amsmath,amsfonts,amssymb}
\usepackage{mathtools}
\usepackage{mathrsfs}
\usepackage[usenames]{color}
\usepackage[all]{xy}
\usepackage{graphicx}
\usepackage{latexsym}
\usepackage{verbatim}

\numberwithin{equation}{section}

\def\Re{{\sf Re}\,}
\def\Im{{\sf Im}\,}

\newcommand{\R}{\mathbb R}

\newcommand{\C}{\mathbb C}

\newcommand{\Co}{\mathcal C}

\newcommand{\N}{\mathbb N}

\def\Re{{\sf Re}\,}
\def\Im{{\sf Im}\,}

\def\id{{\sf id}}

\def\Re{{\sf Re}\,}
\def\Im{{\sf Im}\,}

\newcommand{\eit}{e^{i\theta}}

\def\Re{{\sf Re}\,}
\def\Im{{\sf Im}\,}

\def\1#1{\overline{#1}}
\def\2#1{\widetilde{#1}}
\def\3#1{\widehat{#1}}
\def\4#1{\mathbb{#1}}
\def\5#1{\frak{#1}}
\def\6#1{{\mathcal{#1}}}

\def\Re{{\sf Re}\,}
\def\Im{{\sf Im}\,}

\newcommand{\mcite}[1]{\csname b@#1\endcsname}

\theoremstyle{theorem}

\def\id{{\sf id}}

\def\Re{{\sf Re}\,}
\def\Im{{\sf Im}\,}

\newtheorem{theorem}{Theorem}[section]
\newtheorem{lemma}[theorem]{Lemma}
\newtheorem{proposition}[theorem]{Proposition}
\newtheorem{corollary}[theorem]{Corollary}

\theoremstyle{definition}
\newtheorem{definition}[theorem]{Definition}
\newtheorem{example}[theorem]{Example}

\theoremstyle{remark}
\newtheorem{remark}[theorem]{Remark}

\numberwithin{equation}{section}

\begin{document}

\title[Local visibility and local Gromov hyperbolicity]{Local and global visibility and Gromov hyperbolicity of  domains with respect to the Kobayashi distance}
\author[F. Bracci]{Filippo Bracci$^{\dag}$}
\author[H. Gaussier]{Herv\'e Gaussier$^{\dag\dag}$}
\author[N. Nikolov]{Nikolai Nikolov$^{\dag\dag\dag}$}
\author[P. J. Thomas]{Pascal J. Thomas}

\address{F. Bracci: Dipartimento Di Matematica\\
Universit\`{a} di Roma \textquotedblleft Tor Vergata\textquotedblright\ \\
Via Della Ricerca Scientifica 1, 00133 \\
Roma, Italy}
\email{fbracci@mat.uniroma2.it}

\address{H. Gaussier: Univ. Grenoble Alpes, IF, F-38000 Grenoble, France and CNRS, IF, F-38000 Grenoble, France}\email{herve.gaussier@univ-grenoble-alpes.fr}

\address{N. Nikolov:
Institute of Mathematics and Informatics\\
Bulgarian Academy of Sciences\\
Acad. G. Bonchev Str., Block 8\\
1113 Sofia, Bulgaria\\ \smallskip\newline
Faculty of Information Sciences\\
State University of Library Studies
and Information Technologies\\
69A, Shipchenski prohod Str.\\
1574 Sofia, Bulgaria}
\email{nik@math.bas.bg}

\address{P.J.~Thomas:
Institut de Math\'ematiques de Toulouse; UMR5219 \\
Universit\'e de Toulouse; CNRS \\
UPS, F-31062 Toulouse Cedex 9, France} \email{pascal.thomas@math.univ-toulouse.fr}

\thanks{$^\dag$ Partially supported by PRIN 2017 Real and Complex Manifolds: Topology, Geometry and holomorphic
dynamics, Ref: 2017JZ2SW5, by GNSAGA of INdAM and by the MUR Excellence Department Project MatMod@TOV CUP:E83C23000330006 awarded
to the Department of Mathematics, University of Rome Tor Vergata}
\thanks{$^{\dag\dag}$ Partially supported by ERC ALKAGE}
\thanks{$^{\dag\dag\dag}$ Partially supported by the National Science Fund,
Bulgaria under grant number KP-06-N52/3.}

\date{\today}

\begin{abstract}
We introduce the notion of locally visible and locally Gromov hyperbolic domains in $\C^d$. We prove that a bounded domain in $\C^d$ is locally visible and locally Gromov hyperbolic if and only if it is (globally) visible and Gromov hyperbolic with respect to the Kobayashi distance. This allows to detect, from local information near the boundary, those domains which are Gromov hyperbolic and for which biholomorphisms extend continuously up to the boundary.
\end{abstract}

\maketitle
\tableofcontents
\section{Introduction}
 In dimension one (see, {\sl e.g.}, \cite[Ch. 4]{BCD}), Carath\'eodory's prime end theory gives a precise characterization of continuous extension of Riemann mappings between simply connected domains. Carath\'eodory theory is based on the construction of a compactification with an abstract boundary (whose points are the so-called prime ends) for which every Riemann map extends naturally as a homeomorphism up to the abstract boundary. Then the problem of continuous extension is reduced to the problem of understanding for which simply connected domain the identity map extends as a homeomorphism from the abstract Carath\'eodory boundary to the Euclidean boundary.

Carath\'eodory's theory extends to higher dimension for quasi-conformal maps, but, in general, biholomorphisms are not quasi-conformal. Therefore, in order to study continuous extension of biholomorphisms, one needs a different compactification of domains for which biholomorphisms extend naturally up to the abstract boundary. Several abstract compactifications have been defined by different authors (see, {\sl e.g.} \cite{BH}) for general metric spaces. In several complex variables, since biholomorphisms are isometries for the Kobayashi metric, it is  natural to consider abstract compactifications with respect to such a metric and study their properties. 
Every proper Gromov hyperbolic space (see the definition here under) has an abstract boundary, the Gromov boundary, and a topology, the Gromov topology, which makes  the space together with
that boundary into a compact space. Since isometries naturally extend to the Gromov boundary, 
the fundamental questions now are  characterizing Gromov hyperbolic domains in $\C^d$ and studying the relations between their Euclidean boundary and their Gromov boundary.

In  \cite{BB}, it has been proved that $C^2$-smooth bounded strongly pseudoconvex domains are Gromov hyperbolic and that the identity map naturally extends as a homeomorphism from the Gromov boundary to the Euclidean boundary (thus obtaining homemorphic extension up to the closure of biholomorphisms between strongly pseudoconvex domains). The techniques in \cite{BB} has been further developed in \cite{CL}, obtaining a different proof of Fefferman's extension theorem \cite{Fe}.

In \cite{BG1, BG2, BGZ} this point of view has been used to prove extension of biholomorphisms between Gromov hyperbolic convex domains, proving, for instance, that every convex map from the ball whose image is convex extends as a homeomorphism up to the boundary regardless of the regularity of the image. In \cite{Z1, Z2} it has been proved that a smooth bounded convex domain is Gromov hyperbolic if and only if it is of D'Angelo finite type, while in \cite{Fia}  it is  proved that bounded smooth pseudoconvex domains in $\C^2$ of D'Angelo finite type are Gromov hyperbolic. In \cite{BGZ1}, Gromov hyperbolicity of convex domains is shown to be equivalent to the existence of a complete K\"ahler metric with holomorphic bisectional curvature negatively  pinched close to the boundary, suggesting the idea that Gromov hyperbolicity should be read only from local properties near the boundary.

One of the features of Gromov's compactification  is {\sl visibility}
 with respect to the Gromov boundary. Roughly speaking, visibility with respect to a boundary means that geodesic lines which converge to different points in 
 that boundary bend inside the space. However, visibility with respect to the Euclidan boundary has been exhibited for domains which are not Gromov hyperbolic in \cite{BZ, BM, BNT, Ma}, and turns out to be a key notion for continuous extension of biholomorphisms and Denjoy-Wolff type theorems. In \cite{CMS}, this notion has been extended to embedded submanifolds of $\C^d$.   Note also that a domain may be
Gromov hyperbolic and embedded into $\C^d$ in such a way that it does not enjoy visibility with respect to its Euclidan boundary; see the Remark at the end of Section \ref{sec:c-convex}.

The aim of this paper is to contribute to the previous line of ideas by showing that Gromov hyperbolicity and visibility  with respect to the Euclidean boundary of a bounded domain can be detected just by local properties of the boundary. To be more concrete, we need some definitions.

For a domain $\Omega \subset \mathbb C^d$ , we denote by $k_{\Omega}$ the infinitesimal Kobayashi pseudometric of $\Omega$ and by $K_{\Omega}$ the Kobayashi pseudodistance of $\Omega$.

Let $\Omega \subset \subset \mathbb C^d$ be a complete hyperbolic domain, meaning that $(\Omega,K_{\Omega})$ is a complete metric space. It follows from the Hopf-Rinow theorem that $(\Omega,K_{\Omega})$ is geodesic and thus, every couple of points in $\Omega$ can be joined by a geodesic for $K_\Omega$. If $p, q\in \Omega$, we denote by $[p,q]_\Omega$ any geodesic joining $p$ and $q$.

The metric space $(\Omega, K_\Omega)$ is {\sl Gromov hyperbolic} if every geodesic triangle is $\delta$-thin for some $\delta>0$.

 Any proper geodesic metric space $(X,d)$ which is Gromov hyperbolic can be embedded in a compact space 
$\overline{X}^G:=X\cup \partial_G X$ with a topology that we call the {\sl Gromov topology} whose restriction to $X$ coincides with the natural topology of $X$ (see, {\sl e.g.}, \cite{BH}).

We now turn to the precise definition of visibility.

Let $\Omega\subset \C^d$ be a bounded domain. Let $p, q\in \partial \Omega$,  $p\neq q$. 
We say that the couple $(p,q)$ satisfies the {\sl visibility condition with respect to $K_\Omega$} 
if  there exist a neighborhood 
$V_p$ of $p$ and a neighborhood $V_q$ of $q$ and a compact subset $K$ of $\Omega$ such that $V_p \cap V_q = \emptyset$ and $[x,y]_{\Omega} \cap K \neq \emptyset$ for every $x \in V_p\cap \Omega$, $y \in V_q\cap \Omega$.

We say that $\Omega$ is {\sl visible}  if  every couple of points $p,q \in \partial \Omega$, $p \neq q$, 
satisfies the visibility condition with respect to $K_\Omega$. 

\begin{definition} Let $\Omega\subset\C^d$ be a bounded domain and $p\in \partial \Omega$. We say that
\begin{itemize}
\item $\Omega$ is {\sl locally Gromov hyperbolic at $p$} if  there exists  an open neighborhood $U_p$ of $p$ such that $\Omega \cap U_p$ is connected, $(\Omega \cap U_p,K_{\Omega \cap U_p})$ is  complete hyperbolic and Gromov hyperbolic,
\item $\Omega$ is {\sl locally visible at $p$} if there exists  an open neighborhood $V_p$ of $p$ such that $\Omega \cap V_p$ is connected, $(\Omega \cap V_p,K_{\Omega \cap V_p})$ is  complete hyperbolic and there is an open neighborhood $V_p'\subseteq V_p$ such that every couple of points $q_1,q_2 \in \partial  \Omega\cap V'_p$ satisfies the visibility condition with respect to  $K_{\Omega\cap V_p}$.
\end{itemize}
If $\Omega$ is locally Gromov hyperbolic ({\sl respectively} locally visible) at every $p\in\partial \Omega$ we say that $\Omega$ is {\sl locally Gromov hyperbolic} ({\sl resp.}, {\sl locally visible}).
\end{definition}

The main result of this paper is the following:
\begin{theorem}\label{thm-locglob}
Let $\Omega \subset \subset \mathbb C^d$. Then  the following are equivalent:
\begin{enumerate}
\item $\Omega$ is visible and $(\Omega,K_{\Omega})$ is Gromov hyperbolic,
\item $\Omega$ is locally visible and locally Gromov hyperbolic.
\end{enumerate}
Moreover, if (1) or (2)---and hence both---holds, the identity map ${\sf id}_\Omega: \Omega\to\Omega$ extends continuously as a surjective continuous map  $\overline{\Omega}^G \to \overline{\Omega}$.
\end{theorem}

We point out that, in the previous theorem,  given $p\in \partial \Omega$, the open neighborhood $U_p$ of $p$ such that $(U_p\cap \Omega, K_{U_p\cap \Omega})$ is  Gromov hyperbolic might be different from the open neighborhood $V_p$ of $p$ such that every couple of points of $\partial \Omega\cap V'_p$ satisfies the visibility condition with respect to $K_{V_p\cap \Omega}$. However, it turns out, see Lemma ~\ref{Lem:equiv-vis}, that  also $U_p\cap \Omega$  is locally visible at $p$.

We also point out (see Proposition~\ref{Prop:loc-vis-forever}) that, if  $\Omega$ is complete hyperbolic Gromov hyperbolic and visible, then the visibility condition holds for any open neighborhood $U$ of any boundary point, provided $U\cap \Omega$ is connected and complete hyperbolic.

\begin{definition}
If $(\Omega,K_{\Omega})$ is Gromov hyperbolic and the identity map $\id_\Omega:\Omega\to \Omega$ extends as a homeomorphism from the Gromov closure $\overline{\Omega}^G$ to  the Euclidean closure $\overline{\Omega}$, we say that $\Omega$ is a {\sl Gromov model domain}.
\end{definition}

With this definition at hand, our initial discussion can be summarized (cf. \cite{BNT}) as follows:

\begin{remark}
Let $\Omega_1, \Omega_2\subset \C^d$ be  bounded domains and let $F:\Omega_1\to \Omega_2$ be a biholomorphism. Assume $\Omega_1$ is a Gromov model domain (since $F$ is an isometry for the Kobayashi distance, it follows that $(\Omega_2, K_{\Omega_2})$ is Gromov hyperbolic as well). Then,
\begin{enumerate}
\item  $\Omega_2$ is a Gromov model domain if and only if $F$ extends as a homeomorphism from $\overline{\Omega_1}$ to $\overline{\Omega_2}$.
\item the identity map ${\sf id}_{\Omega_2}: \Omega_2\to\Omega_2$ extends continuously as a surjective continuous map  $\overline{\Omega_2}^G \to \overline{\Omega_2}$ if and only if $F$ extends as a continuous surjective map from $\overline{\Omega_1}$ to $\overline{\Omega_2}$.
\end{enumerate}
\end{remark}

If $D\subset\subset \C^d$ is a domain, following \cite{BNT}, we say that a geodesic line $\gamma:(-\infty,+\infty)\to D$ is a {\sl geodesic loop in $\overline{D}$} if $\gamma$ has the same cluster set $\Gamma$ in $\overline{D}$ at $+\infty$ and $-\infty$. In such a case we say that $\Gamma$ is the {\sl vertex} of the geodesic loop $\gamma$.

Rephrasing  \cite[Thm. 3.3]{BNT} we have that {\sl a complete hyperbolic, Gromov hyperbolic, visible bounded domain is a Gromov model domain if and only if it has no geodesic loops}. For the sake of completeness, in Lemma~\ref{Lem:model-to-vis}, we give a direct proof that any Gromov model domain is visible and has no geodesic loops.

Here we can ``localize'' such a result.

\begin{definition} Let $\Omega\subset\C^d$ be a bounded domain and $p\in \partial \Omega$. We say that
 $\Omega$  {\sl has no local geodesic loops at $p$} if  there exists  an open neighborhood $W_p$ of $p$ such that $\Omega \cap W_p$ is connected, $(\Omega \cap W_p, K_{\Omega \cap W_p})$ is  complete  hyperbolic and there is no geodesic loops for $K_{\Omega \cap W_p}$ whose vertex contains $p$.

If $\Omega$ has no local geodesic loops at every $p\in\partial \Omega$ we say that $\Omega$  {\sl has no local geodesic loops}.
\end{definition}

\begin{theorem}\label{thm-locglob2}
Let $\Omega \subset \subset \mathbb C^d$.
Then $\Omega$ is a Gromov model domain if and only if it is locally Gromov hyperbolic, locally visible and has no local geodesic loops.
\end{theorem}

There is no complete characterization of Gromov model domains, however, it is known that the following are Gromov model domains:
\begin{enumerate}
\item bounded smooth  strongly pseudoconvex domains (see \cite{BB}),
\item bounded smooth convex domains of finite D'Angelo type (see \cite{Z1}),
\item Gromov hyperbolic (with respect to the Kobayashi distance) bounded or unbounded convex domains (see \cite{BGZ}),
\item bounded smooth pseudoconvex domains of finite D'Angelo type in $\C^2$ (see \cite{Fia}),
\item bounded Gromov hyperbolic (with respect to the Kobayashi distance) $\C$-convex domains with Lipschitz boundary (see Proposition~\ref{cc})
\item any domain biholomorphic to a Gromov model domain such that the biholomorphism extends as a homeomorphism up to the boundary.
\end{enumerate}

Theorem~\ref{thm-locglob2}  allows us to ``localize'' the previous list as follows:

\begin{corollary}\label{Cor:localmodel-global}
Let $\Omega \subset \subset \mathbb C^d$ be a domain. Suppose that for every $p\in \partial \Omega$ there exists an open neighborhood $U_p$ of $p$ such that $U_p\cap \Omega$ is a Gromov model domain. Then $(\Omega, K_\Omega)$ is  a Gromov model domain.
\end{corollary}

In particular, if $\Omega\subset\subset \C^d$ is a  domain such that there exists an open covering $\{U_j\}$ of $\partial \Omega$ so that $\Omega\cap U_j$ is biholomorphic to any domain of the previous list of Gromov model domains and the biholomorphism extends as a homeomorphism up to the closure, then $(\Omega, K_\Omega)$ is complete hyperbolic, Gromov hyperbolic, visible and $\overline{\Omega}^G$ is naturally homeomorphic to $\overline{\Omega}$.

The proofs of  Theorem~\ref{thm-locglob}, Theorem~\ref{thm-locglob2} and Corollary~\ref{Cor:localmodel-global} will be given in Section~\ref{S2}.

Finally, in Section~\ref{sec:c-convex},  we consider $\C$-convex
domains with Lipschitz boundary and prove item~(5) of the previous list.

\medskip

\noindent {\sl Acknowledgments}. This work originated from conversations among the authors during the INdAM Workshop ``Gromov hyperbolicity and negative curvature in complex analysis'' held at Palazzone Cortona, Italy  6-10 September 2021.

The authors  warmly thank the two anonymous referees for many useful comments which improved the original manuscript.

\section{Preliminary results}

Let $(X,d)$ be a metric space. If $\gamma:[0,1] \rightarrow X$ is an absolutely continuous curve and $0 \leq s < t \leq 1$, we denote by $l_d(\gamma; [s,t])$ the length of the restriction of $\gamma$ to $[s,t]$.

For an absolutely continuous curve $\gamma:[0,1] \rightarrow X$, we denote by $l_{d}(\gamma;[s,t])$  the length of the curve $\gamma$ on $[s,t]$, which is defined as the total variation $\sup\sum_{i=1}^k d(\gamma(t_{i-1}), \gamma(t_i))$, where the supremum is taken over all partitions $s=t_0<t_1<\ldots<t_k=t$. If $X$ is a manifold and the distance $d$ is  defined  by a inner Finsler metric $F$ on the tangent bundle $TX$ (as it happens for the Kobayashi distance), then $l_{d}(\gamma;[s,t])=\int_s^t F(\gamma(\tau);\gamma'(\tau))d\tau$ and $d(\gamma(s),\gamma(t))\leq l_{d}(\gamma;[s,t])$ for all $0 \leq s < t \leq 1$ (see, {\sl e.g.}, \cite{Kob}).  

Let $A > 1$ and $B > 0$. An absolutely continuous curve $\gamma:[0,1] \rightarrow X$ is called a $(A,B)$-quasi-geodesic if for every $0 \leq s < t \leq 1$, we have :
$$
\frac{1}{A} d(\gamma(s),\gamma(t)) - B\leq l_{d}(\gamma;[s,t]) \leq A d(\gamma(s),\gamma(t)) + B.
$$

Let $(X,d)$ be a  geodesic metric space. For $x,y \in X$, we denote by $[x,y]_X$ a geodesic segment joining $x$ and $y$. A geodesic triangle $T$ is the union of 3 geodesic segments (called {\sl sides}) $T=[x,y]_X \cup [y,z]_X \cup [z,x]_X$ joining 3 points $x,\ y, z \in X$.

A  geodesic metric space $(X,d)$ is {\sl Gromov hyperbolic} if there exists $\delta > 0$
such that  every geodesic triangle $T$
is $\delta$-thin, that is,
every point on a side of $T$ has distance from the union of the other two sides less than or equal to $\delta$.

A  geodesic metric space $(X,d)$ is called {\sl geodesically stable} if for every $A > 1$ and $B > 0$ there exists $M > 0$ with the following property :
If $\gamma : [0, 1] \rightarrow X$ is a $(A,B)$-quasi-geodesic, there exists a geodesic segment $[\gamma(0),\gamma(1)]_X$ such that $\gamma([0,1]) \subset \mathcal N^d_{M}\left([\gamma(0),\gamma(1)]_X\right)$, where if $K\subset \Omega$, we let
\[
\mathcal N^d_M\left(K\right):=\{x \in X : d(x,K) < M\}.
\]

We have (see, for instance, \cite{bon}, Section 3 p.295) :

\begin{theorem}[Geodesic stability] Let $(X,d)$ be a geodesic metric space. Then the following conditions are equivalent:
\begin{itemize}
\item[(a)] $(X,d)$ is Gromov hyperbolic,
\item[(b)] $(X,d)$ is geodesically stable.
\end{itemize}
\end{theorem}

\begin{remark}\label{Rem:converse-inf}
If $(X,d)$ is Gromov hyperbolic then for every $A, A'\geq 1$ and $B, B'\geq 0$ there exists $M=M(A,A', B, B')>0$ such that, if $\gamma:[0,1]\to X$ is a $(A,B)$-quasi-geodesic and $\eta:[0,1]\to X$ is a $(A',B')$-quasi-geodesic with $\gamma(0)=\eta(0)$ and $\gamma(1)=\eta(1)$ we have
$\gamma([0,1])\subset \mathcal N^d_M\left(\eta([0,1]\right))$.
\end{remark}

We first start with the following lemmas. As remarked before, if $\Omega\subset\C^d$ is a domain and $\gamma:[0,1]\to \Omega$ is an absolutely continuous curve, then $l_{K_{\Omega}}(\gamma;[0,1])$ denotes its integrated length with respect to the Kobayashi metric, which coincides with the total variation of the curve with respect to the Kobayashi distance, that is, 
\[
l_{K_{\Omega}}(\gamma;[0,1])=\int_0^1 k_\Omega(\gamma(t);\gamma'(t))dt.
\]

\begin{lemma}\label{lem-loc0}Let $\Omega \subset \subset \mathbb C^d$. Then, for every neighborhood $U$ of $p \in \partial \Omega$ and for every neighborhood $V$ of $p$, $V \subset \subset U$, there exists $A > 0$, such that every absolutely continuous curve $\gamma:[0,1] \rightarrow \Omega \cap V$ satisfies :
$$
l_{K_{\Omega \cap U}}(\gamma;[0,1]) \leq A l_{K_{\Omega}}(\gamma;[0,1]).
$$
In particular, if $[x,y]_{\Omega} \subset V$, then $[x,y]_{\Omega}$ is a $(A,0)$-quasi-geodesic segment for $K_{\Omega \cap U}$.
\end{lemma}

Lemma~\ref{lem-loc0} is a direct consequence of the following Localization Lemma proved by H. Royden, (see also Lemma 2.1 in \cite{fu}) :

\vspace{1mm}
\noindent{\sl {\bf Localization Lemma.}
Let $\Omega \subset \subset \mathbb C^d$ and let $U$ be an open set such that $U\cap\Omega\neq \emptyset$.
Then
$$
k_{\Omega}(z;v) \leq k_{U\cap \Omega}(z;v) \leq \coth\left(K_{\Omega}(z, \Omega \setminus U)\right)k_{\Omega}(z;v),
$$
for any $z \in U$ and $v \in \mathbb C^d$, where $K_{\Omega}(z, \Omega \setminus U):=\inf_{w \in \Omega \setminus U}K_{\Omega}(z,w)$.
}

\vspace{2mm}
\begin{proof}[Proof of Lemma~\ref{lem-loc0}] Since $V\subset\subset U$, it follows that there exists $C>0$ such that $K_{\Omega}(z, \Omega \setminus U) \geq C$ for all $z\in V\cap \Omega$---hence there exists $A>0$ such that $\coth(K_{\Omega}(z, \Omega \setminus U))\leq A$  for all $z\in V\cap \Omega$.

Thus, taking into account that $\gamma([0,1]) \subset V$, by the Localization Lemma, we have:
\begin{equation*}
\begin{split}
l_{K_{\Omega \cap U}}(\gamma;[0,1]) & =   \int_0^1 k_{\Omega \cap U}(\gamma(t);\gamma'(t))dt\\
 & \leq  \int_0^1 \coth(K_{\Omega}(\gamma(t), \Omega \setminus U))k_{\Omega}(\gamma(t);\gamma'(t))dt\\
& \leq A \int_0^1 k_{\Omega}(\gamma(t);\gamma'(t))dt=A l_{K_{\Omega}}(\gamma;[0,1]),
 \end{split}
\end{equation*}
and we are done.
\end{proof}

\begin{lemma}\label{Lem:equiv-vis} Let $\Omega$ be a domain and $p\in\partial \Omega$. Assume $U$ is an open neighborhood of $p$ such that $U\cap \Omega$ is a complete hyperbolic domain and $(U\cap \Omega, K_{U\cap \Omega})$ is Gromov hyperbolic.  If  $\Omega$ is locally visible at $p$ then there exists an open connected neighborhood $V\subset\subset U$ of $p$ such that every couple of points $q_1, q_2\in V\cap\partial\Omega$, $q_1\neq q_2$, satisfies the visibility condition with respect to $K_{U\cap \Omega}$.
\end{lemma}
\begin{proof}
Let $W'\subseteq W$ be  open neighborhoods of $p$ such that $\Omega\cap W$ is complete hyperbolic and every couple $q_1, q_2\in \partial\Omega\cap W'$, $q_1\neq q_2$, satisfies the visibility condition with respect to $K_{\Omega\cap W}$.

Let $V$ be an open, connected neighborhood of $p$ such that $V \subset \subset U \cap W'$. We are going to show that every couple of points $q^1, q^2\in V\cap \partial\Omega$, $q^1\neq q^2$, satisfies the visibility condition with respect to $K_{U\cap\Omega}$.

To this aim, choose an open set $V'$
such that $V \subset \subset V' \subset \subset U \cap W$. By the Localization Lemma, there exist $C_1, C_2 >0$
such that for any $z\in V'$, $v\in \C^n$,
\begin{equation*}
\begin{split}
k_\Omega (z;v) &\leq k_{\Omega \cap U} (z;v) \leq C_1 k_\Omega (z;v)\\
k_\Omega (z;v) & \leq k_{\Omega \cap W} (z;v) \leq C_2 k_\Omega (z;v),
\end{split}
\end{equation*}
therefore, there is  $C>1$ such that
\begin{equation}
\label{metreq}
C^{-1} k_{\Omega \cap W} (z;v) \le k_{\Omega \cap U} (z;v) \le C k_{\Omega \cap W} (z;v).
\end{equation}

Let $q^1\neq q^2 \in V\cap \partial \Omega$, and take sequences  $\{q^j_k\}\subset V\cap \Omega$, such that $\lim_{k\to \infty} q^j_k=q^j$, $j=1,2$.

By the visibility hypothesis
on $W\cap \Omega$, the geodesics $[q^1_k,q^2_k]_{\Omega\cap W}$ intersect a fixed compact set $K$ for all $k$.
Choose $o_k \in [q^1_k,q^2_k]_{\Omega\cap W} \cap K$.

\medskip

{\sl Claim A.}  If $y_k\in \partial V \cap [q^1_k,q^2_k]_{\Omega\cap W}$, then $\{y_k\}$ is relatively compact in $W\cap \Omega$.

\medskip

Indeed, if this were not the case, we can assume, up to extracting subsequences, that $y_k\in [q^1_k,o_k]_{\Omega\cap W}$ and that $\lim_{k\to \infty}y_k=y_0\in \partial V\cap\partial \Omega$. Since $y_0\neq q^1$ by construction, the visibility hypothesis implies that there exists $x_k\in [q^1_k,y_k]_{\Omega\cap W}$ such that $\{x_k\}$ is relatively compact in $\Omega\cap W$. Hence, there exists $T>0$ such that for every $k$,
\[
T>K_{\Omega\cap W}(x_k,o_k)=K_{\Omega\cap W}(x_k,y_k)+K_{\Omega\cap W}(y_k,o_k),
\]
but the right hand side tend to $\infty$ since  $K_{\Omega\cap W}$ is complete, a contradiction and Claim A follows.
\smallskip

Now, if $[q^1_k,q^2_k]_{\Omega\cap W} \not\subset V$, let $q^3_k\in [q^1_k,q^2_k]_{\Omega\cap W} \cap\partial V$ be such that for all $z\in [q^1_k,q^3_k]_{\Omega\cap W}\setminus\{q^3_k\}$, it holds $z\in V$ and let $q^4_k\in [q^1_k,q^2_k]_{\Omega\cap W} \cap\partial V$ be such that for all $z\in [q^4_k,q^2_k]_{\Omega\cap W}\setminus\{q^4_k\}$ it holds $z\in V$.

Since by Claim A, $\{q^3_k\}$ and  $\{q^4_k\}$ are relatively compact in $W\cap \Omega$ and, by construction, they belong to $\partial V$, we can join $q^3_k$ and  $q^4_k$ with a smooth curve $\gamma_k$ such that $\{\gamma_k\}$ is relatively compact  in $V'\cap\Omega$. In particular, for every $k$ the length of $\gamma_k$ is  bounded  by a constant independent of $k$, with respect to any of the metrics $k_\Omega$,
$k_{\Omega\cap W}$, and $k_{\Omega\cap U}$.

Therefore, the curve $\Gamma_k:=[q^1_k,q^3_k]_{\Omega\cap W} \cup \gamma_k \cup [q^4_k,q^2_k]_{\Omega\cap W}$ is  a $(1,B')$-quasi-geodesic
 for $K_{\Omega\cap W}$, for some $B'>0$, and is contained in $V'$. By \eqref{metreq}, there exists $A\geq 1$ such that  for all $k$, $\Gamma_k$ is an
 $(A,B)$-quasi-geodesic for $K_{\Omega \cap U}$. Note that, by construction, $\Gamma_k$ intersects a fixed compact set $K$ in $\Omega\cap V'$.

By the Geodesic Stability Theorem applied to $\Omega \cap U$, any $K_{\Omega \cap U}$-geodesic $\eta_k$ between $q^1_k$
and $q^2_k$ must lie within an $M$-neighborhood of this $(A,B)$-quasi-geodesic, for some $M>0$ depending only on $A$ and $B$. Therefore $\eta_k$ intersects for every $k$ an $M$-neighborhood of $K$ with respect to $K_{\Omega \cap U}$. Since $K_{\Omega \cap U}$ is complete, such a set is also compact in $\Omega\cap U$, and we are done.
\end{proof}

\begin{lemma}\label{lem-vis}
Let $\Omega \subset \subset \mathbb C^d$ be a complete hyperbolic domain  and let $p \in \partial \Omega$.  Let $U_p$ be an open neighborhood of $p$ such that $(U_p\cap\Omega, K_{U_p\cap\Omega})$ is complete hyperbolic and Gromov hyperbolic. Assume there exists $V\subset\subset U_p$  an open neighborhood of $p$ such that every couple of points $q_1,q_2 \in \partial  \Omega\cap V$, $q_1 \neq q_2$, satisfies the visibility condition with respect to $K_{U_p\cap\Omega}$.
 Then for every $W\subset\subset V$ there exists $C > 0$ such that for all $x \in \Omega \cap W$ and $y \in \Omega \setminus \overline{V}$,
$$
\sup_{z \in [x,y]_{\Omega}}d_{\hbox{\tiny Eucl}}(z,\partial \Omega) \geq C,
$$
where, $d_{\hbox{\tiny Eucl}}$ denotes the Euclidean distance.
\end{lemma}

\begin{proof} Arguing by contradiction, we assume that there exist a sequence $\{x_{\nu} \}\subset \Omega \cap W$ and $\{y_{\nu}\} \subset \Omega \setminus\overline{V}$ such that
\begin{equation}\label{eq-boun}
\lim_{\nu \to \infty}\sup_{z \in [x_{\nu},y_{\nu}]_{\Omega}}d_{Eucl}(z,\partial \Omega) = 0.
\end{equation}

Let $y_\nu'\in [x_{\nu},y_{\nu}]_\Omega$ be such that $y'_{\nu}\in V\setminus\overline{W}$ and $[x_{\nu},y'_{\nu}]_\Omega\subset V$. Let $A>0$ be the constant associated to $V, U_p$ and given by  Lemma~\ref{lem-loc0}. Then,  $[x_{\nu},y'_{\nu}]_{\Omega}$ is a $(A,0)$-quasi-geodesic segment for $K_{\Omega \cap U_p}$, for every $\nu$. Since $(\Omega \cap U_p, K_{\Omega \cap U_p})$ is Gromov hyperbolic, by Remark~\ref{Rem:converse-inf} there exists $M > 0$ such that for every $\nu$
\begin{equation}\label{eq-close}
[x_{\nu},y'_{\nu}]_{\Omega \cap U_p} \subset \mathcal N^{K_{\Omega \cap U_p}}_M\left([x_{\nu},y'_{\nu}]_{\Omega}\right).
\end{equation}

Bearing in mind that $\Omega \cap U_p$ satisfies the visibility condition on $\partial\Omega\cap V$, there exists a compact set $K$ of $U_p\cap \Omega$ such that for every $\nu$ we can find  $z_\nu\in [x_{\nu},y'_{\nu}]_{\Omega \cap U_p}$ with $z_\nu\in K$.

By \eqref{eq-close}, for every $\nu$, there exists $z_\nu'\in [x_{\nu},y'_{\nu}]_{\Omega}$ such that $K_{\Omega\cap U_p}(z_\nu',z_\nu)\leq M$.
Since $\Omega\cap U_p$ is complete hyperbolic and $\{z_\nu\}\subset K$, it follows that $\{z'_\nu\}$ is relatively compact in $\Omega\cap U_p$---and hence in $\Omega$, contradicting \eqref{eq-boun}.
\end{proof}

\begin{lemma}\label{lem-loc}
Let $\Omega \subset \subset \mathbb C^d$ be a complete hyperbolic domain  and let $p \in \partial \Omega$.  Let $U_p$ be an open neighborhood of $p$ such that $(U_p\cap\Omega, K_{U_p\cap\Omega})$ is complete hyperbolic and Gromov hyperbolic.  Let $V\subset\subset  U_p$ be an open neighborhood of $p$  such that every couple of points $p,q \in \partial  \Omega\cap V$, $p \neq q$, has the visibility condition with respect to $K_{U_p\cap\Omega}$. Let $A>0$ be the constant given by Lemma~\ref{lem-loc0} associated to $U_p, V$.  Assume $p_{\nu}, q_{\nu} \in \Omega\cap V$ are such that $\lim_{\nu \rightarrow \infty}p_{\nu}=\lim_{\nu \rightarrow \infty}q_{\nu} = p$. If $[p_{\nu},q_{\nu}]_{\Omega} \not\subset V$ for every $\nu$, then there exist $C>0$ and $p'_{\nu}, q'_{\nu} \in [p_{\nu},q_{\nu}]_{\Omega} \cap V$ such that, for every $\nu$,
\begin{enumerate}
\item $[p_{\nu},p'_{\nu}]_{\Omega} \subset V$ and  $[q_{\nu},q'_{\nu}]_{\Omega} \subset V$,
\item $d_{Eucl}(p_\nu',\partial \Omega) > C$ and $d_{Eucl}(q'_\nu,\partial \Omega) > C$,
\item $\cup_\nu [p'_{\nu},q'_{\nu}]_{\Omega\cap U_p} $ is relatively compact in $\Omega\cap U_p$.
\item $\cup_\nu [p'_{\nu},q'_{\nu}]_{\Omega} $ is relatively compact in $\Omega$.
\end{enumerate}
Moreover, there exists $B>0$ such that the curve $[p_{\nu},p'_{\nu}]_{\Omega} \cup [p'_{\nu},q'_{\nu}]_{\Omega\cap U_p}\cup[q_{\nu},q'_{\nu}]_{\Omega} $
is a $(A,B)$-quasi-geodesic for $K_{\Omega\cap U_p}$.
\end{lemma}

\begin{proof} Let $W\subset\subset V$ be an open neighborhood of $p$ such that $\{p_\nu\}, \{q_\nu\}\subset W$. Let $a'_{\nu}, b'_{\nu} \in [p_{\nu},q_{\nu}]_{\Omega} \backslash \overline{W}$ be such that $[p_{\nu},a'_{\nu}]_{\Omega} \subset V$ and $[q_{\nu},b'_{\nu}]_{\Omega} \subset V$. By Lemma~\ref{lem-vis} there exist $p'_\nu \in [p_{\nu},a'_{\nu}]_{\Omega}$ and $q_\nu'\in [q_{\nu},b'_{\nu}]_{\Omega}$ such that (2) holds for some $C>0$ independent of $\nu$.

Statements (3) and (4) follow at once  taking into account that, since $(\Omega, K_\Omega)$ is complete and $\{p_\nu', q_\nu'\}\subset\subset\Omega\cap U_p$, then $\cup_\nu [p'_\nu, q'_\nu]_\Omega$ and $\cup_\nu [p'_\nu, q'_\nu]_{\Omega\cap U_p}$ are relatively compact in $\Omega$ and $\Omega\cap U_p$ respectively.

Finally, since $[p_\nu, q_\nu]_\Omega$ is a geodesic for $K_\Omega$ and
\[
[p_\nu, q_\nu]_\Omega=[p_{\nu},p'_{\nu}]_{\Omega} \cup [p'_{\nu},q'_{\nu}]_{\Omega}\cup[q_{\nu},q'_{\nu}]_{\Omega},
\]
it follows by (3) and (4) that there exists $B>0$ such that
\[
l_{\Omega\cap U_p}([p'_{\nu},q_\nu']_{\Omega\cap U_p})=K_{\Omega\cap U_p}(p'_{\nu},q_\nu') \leq B \leq AK_{\Omega}(p'_{\nu},q_\nu')+B=Al_{\Omega}([p'_{\nu},q'_{\nu}]_{\Omega})+B,
\]
and such that for every $\xi\in [p_{\nu},p'_{\nu}]_{\Omega}$  and $\zeta\in [p'_{\nu},q'_\nu]_{\Omega\cap U_p}$,
\[
|K_{\Omega\cap U_p}(\xi, p'_\nu)-K_{\Omega\cap U_p}(\xi, \zeta)|\leq K_{\Omega\cap U_p}(p'_\nu, \zeta)\leq K_{\Omega\cap U_p}(p'_\nu, q_\nu')\leq \frac{B}{A+1}.
\]
Using Lemma~\ref{lem-loc0} for $[p_{\nu},p'_{\nu}]_{\Omega}$ and $[q_{\nu},q'_{\nu}]_{\Omega}$ and the previous inequalities, it is easy to show that $[p_{\nu},p'_{\nu}]_{\Omega} \cup [p'_{\nu},q'_{\nu}]_{\Omega\cap U_p}\cup[q_{\nu},q'_{\nu}]_{\Omega} $ is a $(A,B)$-quasi-geodesics for $K_{\Omega\cap U_p}$.
Indeed,  if $\xi\in [p_{\nu},p'_{\nu}]_{\Omega}$  and $\zeta\in  [p'_{\nu},q'_{\nu}]_{\Omega\cap U_p}$, let us denote by $l_{\Omega\cap U_p}([\xi, \zeta])$ the length of the curve $[\xi,p'_{\nu}]_{\Omega} \cup [p'_{\nu},\zeta]_{\Omega\cap U_p}$ with respect to $K_{\Omega\cap U_p}$, and similarly denote by  $l_{\Omega}([\xi, \zeta])$ its length with respect to $K_\Omega$. Hence, 
\begin{equation*}
\begin{split}
l_{\Omega\cap U_p}([\xi, \zeta])
&=l_{\Omega\cap U_p}([\xi, p'_\nu]_{\Omega})
+l_{\Omega\cap U_p}([p'_\nu, \zeta]_{\Omega\cap U_p})\\
&\leq A K_{\Omega\cap U_p}(\xi, p'_\nu)+K_{\Omega\cap U_p}(p'_\nu, \zeta)\leq A K_{\Omega\cap U_p}(\xi,\zeta)+(A+1)\frac{B}{A+1}\\
&= A K_{\Omega\cap U_p}(\xi,\zeta)+B.
\end{split}
\end{equation*}
 If $\xi\in [p_{\nu},p'_{\nu}]_{\Omega}$  and $\zeta\in  [q_{\nu},q'_{\nu}]_{\Omega}$, we have
\begin{equation*}
\begin{split}
l_{\Omega\cap U_p}([\xi, \zeta])
&=l_{\Omega\cap U_p}([\xi, p'_\nu]_{\Omega})+l_{\Omega\cap U_p}([p'_\nu, q_\nu']_{\Omega\cap U_p})+l_{\Omega\cap U_p}([q_\nu', \zeta]_{\Omega})\\
&\leq A l_{\Omega}([\xi, p'_\nu]_{\Omega})+Al_{\Omega}([p'_\nu, q_\nu']_{\Omega})+B+Al_{\Omega}([q_\nu', \zeta]_{\Omega})\\
&=A K_{\Omega}(\xi,\zeta)+B\leq A K_{\Omega\cap U_p}(\xi,\zeta)+B,
\end{split}
\end{equation*}
and we are done.
\end{proof}

\begin{lemma}\label{loc-glob-complete}
Let $\Omega \subset \subset \mathbb C^d$ be a  domain. If $\Omega$ is locally  visible then $(\Omega, K_\Omega)$ is complete hyperbolic.
\end{lemma}
\begin{proof}
We need to show that if $\{z_k\}\subset \Omega$ is a sequence such that $K_\Omega(z_0, z_k)\leq C$ for all $k$ and for some $C>0$ then $\{z_k\}$ is relatively compact in $\Omega$. Suppose this is not the case and assume that $\{z_k\}$ converges to $p\in \partial \Omega$. Let $V$ be an open neighborhood of  $p$ such that $(\Omega\cap V, K_{\Omega\cap V})$ is complete hyperbolic and let $V'\subseteq V$ be  an open neighborhood of $p$ such that every couple of distinct points of $\partial \Omega\cap V'$ satisfies the visibility condition with respect to $K_{\Omega\cap V}$.  Let $W\subset\subset V'$ be an open neighborhood of $p$.

We can  assume that $\{z_k\}\subset W$. Fix $\epsilon>0$. Let $\gamma_k:[0,1]\to \Omega$ be a smooth curve  such that $\gamma_k(0)=z_0$, $\gamma_k(1)=z_k$ and
\[
l_{K_{\Omega}}(\gamma_k;[0,1])\leq K_\Omega(z_0, z_k)+\epsilon\leq C+\epsilon.
\]
If $\gamma_k([0,1])\subset W$, then by Lemma~\ref{lem-loc0} there exists $A>0$ such that
\[
K_{\Omega\cap V}(z_0,z_k)\leq l_{K_{\Omega\cap V}}(\gamma_k;[0,1])\leq A l_{K_{\Omega}}(\gamma_k;[0,1])\leq A(C+\epsilon),
\]
and, since $K_{\Omega\cap V}$ is complete, we have a contradiction.

Therefore, we can find $0<t_k^0<t_k^1<1$ such that $\gamma_k(t)\in W$ for all $t\in [0,t_k^0)\cup (t_k^1,1]$ and $\gamma_k(t_k^0), \gamma_k(t_k^1)\in \partial W\cap \Omega$ for all $k$.

If $\{\gamma_k(t_k^0)\}\cup\{\gamma_k(t_k^1)\}$ is relatively compact in $\Omega$ (and hence by construction in $V\cap \Omega$), it follows that $K_{\Omega\cap V}(\gamma_k(t_k^0),\gamma_k(t_k^1))\leq T$ for some fixed $T>0$ and for all $k$. Thus, arguing as before, we have
\begin{equation*}
\begin{split}
K_{\Omega\cap V}(z_0,z_k)&\leq K_{\Omega\cap V}(z_0,\gamma_k(t_k^0))+
K_{\Omega\cap V}(\gamma_k(t_k^0),\gamma_k(t_k^1))+K_{\Omega\cap V}(\gamma_k(t_k^1), z_k)\\&\leq  l_{K_{\Omega\cap V}}(\gamma_k;[0,t_k^0])+l_{K_{\Omega\cap V}}(\gamma_k;[t_k^1,1])+T\\&\leq Al_{K_{\Omega}}(\gamma_k;[0,t_k^0])+Al_{K_{\Omega}}(\gamma_k;[t_k^1,1])+T\\& \leq 
Al_{K_{\Omega}}(\gamma_k;[0,1])+T\leq A(C+\epsilon)+T,
\end{split}
\end{equation*}
again, a contradiction.

Hence, $\{\gamma_k(t_k^0)\}\cup\{\gamma_k(t_k^1)\}$ is not relatively compact in $\Omega$, and we can assume first that $\{\gamma_k(t_k^0)\}$ converges to some $q\in \partial\Omega\cap \partial W$. In particular, $q\neq p$. Thus, as before,
\[
K_{\Omega\cap V}(z_0, \gamma_k(t_k^0))\leq l_{K_{\Omega\cap V}}(\gamma_k;[0,t_k^0])\leq A l_{K_{\Omega}}(\gamma_k;[0,1])\leq A(C+\epsilon),
\]
again a contradiction.

Finally, we are left to consider the case $\{\gamma_k(t_k^1)\}$ converges to some $q\in \partial\Omega\cap \partial W$. Arguing as before, we see that
\[
K_{\Omega\cap V}(z_k, \gamma_k(t_k^1))\leq A(C+\epsilon).
\]
Now, by visibility condition, for every $k$ there exists  $x_k\in [z_k, \gamma_k(t_k^1)]_{\Omega\cap V}$ such that $\{x_k\}$ is relatively compact in $V\cap \Omega$, but since
\[
K_{\Omega\cap V}(z_k, x_k)+K_{\Omega\cap V}(x_k, \gamma_k(t_k^1))=K_{\Omega\cap V}(z_k, \gamma_k(t_k^1))\leq A(C+\epsilon),
\]
we have again a contradiction.
\end{proof}

We have

\begin{corollary}\label{loc-visibility}
Let $\Omega \subset \subset \mathbb C^d$ be a domain which is locally Gromov hyperbolic and locally visible.  Then $(\Omega, K_\Omega)$ is complete hyperbolic and $\Omega$ is visible.
\end{corollary}
\begin{proof}
By Lemma~\ref{loc-glob-complete}, $(\Omega, K_\Omega)$ is complete hyperbolic.

In order to prove that $\Omega$ is visible, let  $p, q\in \partial \Omega$, $p\neq q$. Let $U_p$ be an open neighborhood of $p$ such that $U_p\cap\Omega$ is a complete hyperbolic domain and $(U_p\cap\Omega, K_{U_p\cap\Omega})$ is Gromov hyperbolic.  By Lemma~\ref{lem-vis} there exists an open  neighborhood $V\subset\subset U_p$  of $p$ such that every couple of points $q_1,q_2\in V\cap \partial \Omega$, $q_1\neq q_2$, satisfies the visibility condition with respect to $K_{U_p\cap\Omega}$. 

Since $q\neq p$, we can choose $V$ in such a way that $q\not\in \overline{V}$. Let $W\subset\subset V$ be an open neighborhood of $p$, let $C>0$ be the constant given by Lemma~\ref{lem-vis} and let $K:=\{\zeta\in\Omega  : d_{Eucl}(\zeta,\partial\Omega)\geq C/2\}$. Let $Q$ be an open neighborhood of $q$ such that $Q\cap \overline{V}=\emptyset$. Hence,  by Lemma~\ref{lem-vis},   if $x\in \Omega\cap W$ and $y\in \Omega\cap Q$ then $[x,y]_\Omega\cap K\neq\emptyset$, and therefore $p,q$ satisfies the visibility condition. 
\end{proof}

As we stated in the introduction, whenever a domain is Gromov hyperbolic and visible, the visibility condition holds locally:

\begin{proposition}\label{Prop:loc-vis-forever}
Let $\Omega\subset\subset \C^d$ be a bounded domain. Suppose that $(\Omega, K_\Omega)$ is complete hyperbolic and Gromov hyperbolic and that $\Omega$ is visible. Let $p\in \partial\Omega$ and let $W$ be an open neighborhood  of $p$ such that $\Omega\cap W$ is connected and complete hyperbolic\footnote{This last property holds if, for example, $W$ is complete hyperbolic.}.
 Then every couple of distinct points in $\partial\Omega\cap W$ satisfies the visibility condition with respect to $K_{\Omega\cap W}$.
\end{proposition}

\begin{proof}
We argue by contradiction.  Suppose $q^1 \neq q^2 \in W\cap \partial \Omega$ and
there are sequences $q^1_k\to q^1$, $q^2_k\to q^2$
such that for any compactum $K \subset \Omega\cap W$,
 the $K_{\Omega\cap W}$-geodesic
$[q^1_k,q^2_k]_{{\Omega\cap W}}$ avoids $K$ for $k$ large enough.

Let $V$ be a neighborhood of $q^1$ such that $V \subset \subset W$, and $q^2 \notin V$.
We may assume that $q^1_k \in V$ and $q^2_k\not\in V$ for all $k$. There exists $q'_k \in \partial V \cap [q^1_k,q^2_k]_{{\Omega\cap W}}$
such that $[q^1_k,q'_k]_{{\Omega\cap W}} \setminus \{q'_k\} \subset V$. By our assumption, passing to a subsequence, we
may assume that $q'_k\to q' \in \partial V \cap \partial (\Omega\cap W)$. But $\partial V \cap \partial W = \emptyset$,
so in fact $q' \in \partial V \cap \partial \Omega$.

By Lemma~\ref{lem-loc0}, $[q^1_k,q'_k]_{{\Omega\cap W}}$ is an $(A,0)$-quasi-geodesic for $K_\Omega$.
By the Geodesic Stability Theorem,
 $[q^1_k,q'_k]_{\Omega} \subset \mathcal N_M ([q^1_k,q'_k]_{{\Omega\cap W}})$ for some $M>0$,
independent of $k$. By the
visibility condition for $(\Omega, K_\Omega)$, there is a compactum $L\subset \Omega$ such that $L\cap [q^1_k,q'_k]_{\Omega}\neq \emptyset$.
Since $\Omega$ is complete hyperbolic, $\mathcal N_M (L)$ is relatively compact in $\Omega$, and
$L':=\overline{\mathcal N_M (L)\cap V}$ is a compactum in $W\cap \Omega$. But $L'\cap [q^1_k,q'_k]_{{\Omega\cap W}}\neq \emptyset$
for all $k$, which contradicts our assumption.
\end{proof}

\section{Proof of the main results}\label{S2}

In this section we give the proofs of Theorem~\ref{thm-locglob}, Theorem~\ref{thm-locglob2} and Corollary~\ref{Cor:localmodel-global}.

\subsection{Proof of Theorem~\ref{thm-locglob}}

\begin{proof}[Proof of Theorem~\ref{thm-locglob}] Since (1) implies (2) is trivial, we prove that (2) implies (1).

By Corollary~\ref{loc-visibility}, $\Omega$ is complete hyperbolic and visible. Thus we have to show that $(\Omega, K_\Omega)$ is Gromov hyperbolic.

By Lemma~\ref{Lem:equiv-vis}, every $p\in\partial \Omega$ has  open neighborhoods $U'_p\subseteq U_p$ such that $(\Omega\cap U_p, K_{\Omega\cap U_p})$ is complete hyperbolic and Gromov hyperbolic and every couple of distinct points in $\partial\Omega\cap U'_p$ satisfies the visibility condition with respect to $K_{\Omega\cap U_p}$.

We assume, to get a contradiction, that for every integer $\nu \geq 1$, there exist $x_{\nu},\ y_{\nu},\ z_{\nu} \in \Omega$ and $a_{\nu} \in [x_{\nu},y_{\nu}]_{\Omega}$ such that for every $\nu$
\begin{equation}\label{eq-abs}
K_{\Omega}(a_{\nu},[x_{\nu},z_{\nu}]_{\Omega} \cup [y_{\nu},z_{\nu}]_{\Omega}) \geq \nu.
\end{equation}

Up to extracting a subsequence, we may assume that there exist points $x_{\infty}$, $y_{\infty}$, $z_{\infty}$ and $a_{\infty} \in \overline{\Omega}$ such that
$$
\lim_{\nu \rightarrow \infty}x_{\nu} = x_{\infty},\ \lim_{\nu \rightarrow \infty}y_{\nu} = y_{\infty},\ \lim_{\nu \rightarrow \infty}z_{\nu} = z_{\infty},\ \lim_{\nu \rightarrow \infty}a_{\nu} = a_{\infty}.
$$
We will consider different cases, depending on the respective locations of the points $x_{\infty}, \ y_{\infty}, \ z_{\infty}, \ a_{\infty}$.

Either $x_\infty$ or $y_\infty$ (or both) belongs to $\partial \Omega$. Indeed, if $x_{\infty} \in \Omega$ and $y_{\infty} \in \Omega$, then there is a compact subset $K$ of $\Omega$ such that $[x_{\nu},y_{\nu}]_{\Omega} \subset K$ for every $\nu \geq 1$. It follows that $\sup_{\nu \geq 1}K_{\Omega}(a_{\nu},\{x_{\nu},y_{\nu}\}) < \infty$. This contradicts (\ref{eq-abs}). We may then assume that $x_{\infty} \in \partial \Omega$.

\vspace{1mm}
\noindent{\bf Case I. $a_{\infty} \in \Omega,\ y_{\infty} \in \Omega$.} Then, by the same argument as above, we have: $\sup_{\nu \geq 1}K_{\Omega}(a_{\nu},y_{\nu}) < \infty$. This contradicts (\ref{eq-abs}).

\vspace{2mm}
\noindent{\bf Case II. $x_{\infty} \neq y_{\infty}, a_{\infty} \neq x_{\infty}, a_{\infty} \neq y_{\infty}$.}

\vspace{1mm}
\noindent{\bf Subcase II.1. $y_{\infty} \in \Omega$.} From Case I, we know that $a_\infty \in \partial \Omega$. According to Lemma~\ref{lem-vis}, there exists $C>0$ and, for every $\nu \geq 1$ there is $a'_{\nu} \in [x_{\nu},a_{\nu}]_{\Omega}$ such that
$$
\inf_{\nu >>1} d_{Eucl}(a'_{\nu},\partial \Omega) \geq C.
$$
In particular there is a compact subset $K'$ of $\Omega$ such that for every $\nu \geq 1$, $a_{\nu} \in K'$, $y_{\nu} \in K'$. It follows
$$
\sup_{\nu \geq 1}K_{\Omega}(a_{\nu},y_{\nu}) \leq \sup_{\nu \geq 1}K_{\Omega}(a'_{\nu},y_{\nu}) < \infty,
$$
which contradicts  (\ref{eq-abs}).

\vspace{1mm}
\noindent{\bf Subcase II.2. $y_{\infty} \in \partial \Omega$.} We may assume that $z_{\infty} \neq x_{\infty}$ (otherwise, if $z_\infty=x_\infty\neq y_\infty$ we repeat the argument switching $x_\infty$ with $y_\infty$). It follows from Lemma~\ref{lem-vis} that there exists $C > 0$ and, for every $\nu \geq 1$, there exists $z'_{\nu} \in [x_{\nu},z_{\nu}]_{\Omega}$ such that
$$
\inf_{\nu \geq 1}d_{Eucl}(z'_{\nu},\partial \Omega) \geq C.
$$

Notice that if $a_{\infty} \in \Omega$ then $\sup_{\nu \geq 1}K_{\Omega}(a_{\nu},z'_{\nu}) < \infty$, which contradicts (\ref{eq-abs}). Hence, we necessarily have  $a_{\infty} \in \partial \Omega$.

It also follows from Lemma~\ref{lem-vis} that there exists $C' > 0$ such that for every $\nu \geq 1$, there exist $x'_{\nu}\in [x_{\nu},a_{\nu}]_\Omega$ and $y'_{\nu}\in [a_\nu, y_{\nu}]_\Omega$ satisfying $d_{Eucl}(x'_{\nu},\partial \Omega) \geq C',\ d_{Eucl}(y'_{\nu},\partial \Omega) \geq C'$. In particular, there exists $c> 0$ such that $\sup_{\nu \geq 1}K_{\Omega}(x'_{\nu},y'_{\nu}) \leq c$. Hence,
\[
c\geq K_{\Omega}(x'_{\nu},y'_{\nu})=K_\Omega(x'_\nu, a_\nu)+K_\Omega(a_\nu, y_\nu'),
\]
and since $a_\infty\in\partial \Omega$ and $\Omega$ is complete, we obtain a contradiction.

\vspace{2mm}
From now on, let  $V_{y_{\infty}} \subset \subset U'_{y_{\infty}}$, be an open neighborhood of $y_{\infty}$.

\vspace{2mm}

\noindent{\bf Case III. $x_{\infty} \neq y_{\infty}$, $a_{\infty} = y_{\infty} \in \partial \Omega$.}

 We consider two subcases.

\vspace{1mm}
\noindent{\bf Subcase III.1}. $z_{\infty}\neq y_{\infty}$. Fix an open neighborhood $W\subset\subset V_{y_\infty}$ of $y_\infty$. Up to starting from an index $\nu_0>1$, we can assume that $a_\nu, y_\nu\in W$ for all $\nu\geq 1$. We first notice that for every $\nu \geq 1$, there exists $p_{\nu} \in [y_{\nu},z_{\nu}]_{\Omega} \cap (V_{y_{\infty}}\setminus\overline{W})$, such that $[y_{\nu},p_{\nu}]_{\Omega} \subset \Omega \cap V_{y_{\infty}}$. According to Lemma~\ref{lem-vis} there exists $C>0$ such that for every $\nu \geq 1$ there exists $z'_{\nu} \in [y_{\nu},p_{\nu}]_{\Omega}$ with $d_{Eucl}(z'_{\nu},\partial \Omega) \geq C$.

Since $y_\infty=a_{\infty} \neq x_{\infty}$, by the same token as before, there also exists, for every $\nu \geq 1$, a point $a'_{\nu} \in [y_{\nu},x_{\nu}]_{\Omega}$ such that $[a'_{\nu},y_{\nu}]_{\Omega} \subset \Omega\cap V_{y_{\infty}}$ and $d_{Eucl}(a'_{\nu},\partial \Omega) \geq C$.

We claim that $a_{\nu} \in [a'_{\nu},y_{\nu}]_\Omega$. Assume this is not the case. Since $a_\infty=y_\infty$ and $x_\infty\neq y_\infty$, by the same token as above applied to $[x_\nu, a_\nu]_\Omega$, we can find $a_\nu''\in [x_\nu, a_\nu]_\Omega$ such that $[a_\nu'', a_\nu]_\Omega\subset  \Omega\cap V_{y_{\infty}}$ and $d_{Eucl}(a''_{\nu},\partial \Omega) \geq C$. Hence, $a_\nu\in [a_\nu'', a_\nu']_\Omega$. But then, since $\{a_\nu'\}$ and $\{a_\nu''\}$ are relatively compact in $\Omega$---say they are contained in the compact subset $K$  of $\Omega$---it follows that there exists $C'>0$ such that for all $\nu\geq 1$,
\[
C'> K_\Omega(a_\nu'', a_\nu')=K_\Omega(a_\nu'', a_\nu)+K_\Omega(a_\nu, a_\nu')\geq 2K_\Omega(a_\nu, K).
\]
Since $K_\Omega$ is complete and $\lim_{\nu\to \infty}K_\Omega(a_\nu, K)=\infty$, we have a contradiction and the claim follows.

Now, from Lemma~\ref{lem-loc0}, there exists $A>0$ such that for every $\nu \geq 1$ :

\vspace{1mm}
$\bullet$ $[y_{\nu},z'_{\nu}]_{\Omega}$ is a $(A,0)$-quasi-geodesic for $K_{\Omega \cap U_{y_{\infty}}}$,

\vspace{1mm}
$\bullet$ $[a'_{\nu},y_{\nu}]_{\Omega}$ is a $(A,0)$-quasi-geodesic for $K_{\Omega \cap U_{y_{\infty}}}$.

\vspace{2mm}
In particular, the curve $[y_{\nu},z'_{\nu}]_{\Omega} \cup [z'_{\nu},a'_{\nu}]_{\Omega \cap U_{y_{\infty}}} \cup [a'_{\nu},y_{\nu}]_{\Omega}$ is a $(A,0)$-quasi-geodesic triangle for $K_{\Omega \cap U_{y_{\infty}}}$.

\vspace{2mm}
By assumption, $(\Omega \cap U_{y_{\infty}}, K_{\Omega \cap U_{y_{\infty}}})$ is Gromov hyperbolic, and by the previous claim $a_{\nu} \in [a'_{\nu},y_{\nu}]$. Hence, it follows from Remark~\ref{Rem:converse-inf} that there exists $M>0$ such that for every $\nu \geq 1$ :
$$
K_{\Omega}(a_{\nu},[y_{\nu},z'_{\nu}]_{\Omega} \cup[z'_{\nu},a'_{\nu}]_{\Omega \cap U_{y_{\infty}}}) \leq K_{\Omega \cap U_{y_{\infty}}}(a_{\nu},[y_{\nu},z'_{\nu}]_{\Omega} \cup [z'_{\nu},a'_{\nu}]_{\Omega \cap U_{y_{\infty}}}) \leq M.
$$

Since $\lim_{\nu \rightarrow \infty}K_{\Omega}(a_{\nu},[y_{\nu},z'_{\nu}]_{\Omega}) = +\infty$, we obtain, for every $\nu >> 1$ :
$$
K_{\Omega}(a_{\nu},[z'_{\nu},a'_{\nu}]_{\Omega \cap U_{y_{\infty}}}) \leq M.
$$
However, since $d_{Eucl}(a'_{\nu},\partial \Omega) \geq C$ and $d_{Eucl}(z'_{\nu},\partial \Omega) \geq C$, there exists a compact subset $K'$ of $\Omega \cap U_{y_{\infty}}$ such that $[z'_{\nu},a'_{\nu}]_{\Omega \cap U_{y_{\infty}}} \subset K'$, for every $\nu \geq 1$. Since $\lim_{\nu \rightarrow \infty}a_{\nu} = a_{\infty} \in \partial \Omega$ and $(\Omega,K_{\Omega})$ is complete, this is a contradiction.

\vspace{1mm}
\noindent{\bf Subcase III.2. $z_{\infty} = y_{\infty}$.}

By Lemma~\ref{lem-vis}, there exists, for every $\nu \geq 1$, a point $z'_{\nu} \in [x_{\nu},z_{\nu}]_{\Omega}$ such that $[z_{\nu},z'_{\nu}]_{\Omega} \subset V_{y_{\infty}}$ and $d_{Eucl}(z'_{\nu},\partial \Omega) \geq C$.  Equivalently, there exists a point $a'_{\nu} \in [a_{\nu},x_{\nu}]_{\Omega}$ such that $[y_{\nu},a'_{\nu}]_{\Omega} \subset V_{y_{\infty}}$ and $d_{Eucl}(a'_{\nu},\partial \Omega) \geq C$. Moreover, as before, $a_\nu\in [y_{\nu},a'_{\nu}]_{\Omega}$.
Then as in Subcase III.1, we have :

\vspace{1mm}
- $[y_{\nu},a'_{\nu}]_{\Omega}$ is a-$(A,0)$ quasi-geodesic for $K_{\Omega \cap U_{y_{\infty}}}$,

\vspace{1mm}
- $[z_{\nu},z'_{\nu}]_{\Omega}$ is a $(A,0)$-quasi-geodesic for $K_{\Omega \cap U_{y_{\infty}}}$.

\vspace{2mm}
\noindent $\bullet$ We assume first that for every $\nu >>1$, $[y_{\nu},z_{\nu}] \subset V_{y_{\infty}}$.

Consider, for every $\nu \geq 1$, the curve $\mathcal C_{\nu}:=[y_{\nu},a'_{\nu}]_{\Omega} \cup [a'_{\nu},z'_{\nu}]_{\Omega \cap U_{y_{\infty}}} \cup [z'_{\nu},y_{\nu}]_{\Omega \cap U_{y_{\infty}}}$.

Since $\lim_{\nu \rightarrow \infty}a_{\nu} = a_{\infty}$ and the set $\cup_{\nu \geq 1}[a'_{\nu},z'_{\nu}]_{\Omega \cap U_{y_{\infty}}}$ is relatively compact in $\Omega$, we have  $\lim_{\nu \rightarrow \infty}K_{\Omega \cap U_{y_{\infty}}}(a_{\nu},[a'_{\nu},z'_{\nu}]_{\Omega \cap U_{y_{\infty}}}) = +\infty$.

Since $(\Omega \cap U_{y_{\infty}}, K_{\Omega \cap U_{y_{\infty}}})$ is Gromov hyperbolic, it follows from the Geodesic stability Theorem that there exists $M>0$ such that for every $\nu \geq 1$, there exists a point $z''_{\nu} \in [z'_{\nu},y_{\nu}]_{\Omega \cap U_{y_{\infty}}}$ such that $K_{\Omega \cap U_{y_{\infty}}}(a_{\nu},z''_{\nu}) \leq M$.

Consider now the curve $\mathcal C'_{\nu}:=[y_{\nu},z_{\nu}]_{\Omega} \cup [z_{\nu},z'_{\nu}]_{\Omega} \cup [z'_{\nu},y_{\nu}]_{\Omega \cap U_{y_{\infty}}}$. Increasing  $M$ if necessary, we have, for every $\nu \geq 1$:
$$
K_{\Omega \cap U_{y_{\infty}}}(z''_{\nu},[y_{\nu},z_{\nu}]_{\Omega} \cup [z_{\nu},z'_{\nu}]_{\Omega}) \leq M.
$$
By the triangle inequality this implies
$$
\sup_{\nu \geq 1}K_{\Omega \cap U_{y_{\infty}}}(a_{\nu},[y_{\nu},z_{\nu}]_{\Omega} \cup [z_{\nu},z'_{\nu}]_{\Omega}) \leq 2M,
$$
contradicting  (\ref{eq-abs}).

\vspace{2mm}
\noindent $\bullet$ We assume now, by extracting a subsequence, that for every $\nu \geq 1$, $[y_{\nu},z_{\nu}] \not\subset V_{y_{\infty}}$.

Extracting again a subsequence if necessary, it follows from Lemma~\ref{lem-loc} that there is, for every $\nu \geq 1$, a point $y'_{\nu} \in [y_{\nu},z_{\nu}]_{\Omega}$ such that $[y_{\nu},y'_{\nu}]_{\Omega} \subset V_{y_{\infty}}$ and $d_{Eucl}(y'_{\nu},\partial \Omega) \geq C$.

In particular, by Lemma~\ref{lem-loc0}, the curve
$$
\mathcal C_{\nu}:=[y_{\nu},a'_{\nu}]_{\Omega} \cup [a'_{\nu},y'_{\nu}]_{\Omega \cap U_{y_{\infty}}} \cup [y'_{\nu},y_{\nu}]_{\Omega}
$$
is a $(A,0)$-quasi-geodesic triangle for $K_{\Omega \cap U_{y_{\infty}}}$.

Since $(\Omega \cap U_{y_{\infty}}, K_{\Omega \cap U_{y_{\infty}}})$ is Gromov hyperbolic, then changing $M$ if necessary, Condition~(\ref{eq-abs}) implies that for all $\nu\geq 1$
$$
K_{\Omega}(a_{\nu},[a'_{\nu},y'_{\nu}]_{\Omega \cap U_{y_{\infty}}} )\leq K_{\Omega \cap U_{y_{\infty}}}(a_{\nu},[a'_{\nu},y'_{\nu}]_{\Omega \cap U_{y_{\infty}}}) \leq M.
$$
This contradicts the fact that the set $\cup_{\nu \geq 1}[a'_{\nu},y'_{\nu}]_{\Omega \cap U_{y_{\infty}}}$ is relatively compact in $\Omega$ and $\lim_{\nu \rightarrow \infty}a_{\nu} = a_{\infty} \in \partial \Omega$.

\vspace{2mm}
\noindent{\bf Case IV. $x_{\infty} = y_{\infty} = z_{\infty}$.}

\vspace{1mm}
\noindent{$\bullet$} Assume that for sufficiently large $\nu$, $[x_{\nu},y_{\nu}]_{\Omega}$, $[y_{\nu},z_{\nu}]_{\Omega}$ and $[x_{\nu},z_{\nu}]_{\Omega}$ are contained in $V_{y_{\infty}}$. It follows then from Lemma~\ref{lem-loc0} that they are $(A,0)$-quasi-geodesic segments for $K_{\Omega \cap U_{x_{\infty}}}$. Hence, since $(\Omega \cap U_{x_{\infty}}, K_{\Omega \cap U_{x_{\infty}}})$ is Gromov hyperbolic, it follows from the geodesic stability Theorem that there exists $M > 0$ such that
$$
\sup_{\nu >>1}K_{\Omega}(a_{\nu},[x_{\nu},z_{\nu}]_{\Omega} \cup [y_{\nu},z_{\nu}]_{\Omega}) \leq \sup_{\nu >>1}K_{\Omega \cap U_{y_{\infty}}}(a_{\nu},[x_{\nu},z_{\nu}]_{\Omega} \cup [y_{\nu},z_{\nu}]_{\Omega}) \leq M.
$$
This contradicts (\ref{eq-abs}).

\vspace{1mm}
\noindent{$\bullet$} Assume, up to extracting a subsequence, that for every $\nu \geq 1$, $[x_{\nu},y_{\nu}]_{\Omega} \not\subset V_{y_{\infty}}$. It follows from Lemma~\ref{lem-loc} that there exist $B>0$ and points $x'_{\nu}, x''_{\nu} \in [x_{\nu},y_{\nu}]_{\Omega}$ such that
$$
\mathcal C_{\nu}:=[x_{\nu},x'_{\nu}]_{\Omega} \cup [x'_{\nu},x''_{\nu}]_{\Omega \cap U_p} \cup [x''_{\nu},y_{\nu}]_{\Omega}
$$
is a $(A,B)$-quasi-geodesic  for $K_{\Omega \cap U_p}$ and there exists $M>0$ such that
$$
[x_{\nu},y_{\nu}]_{\Omega} \subset \mathcal N^{K_{\Omega}}_M\left(\mathcal C_{\nu}\right).
$$
In particular, $K_{\Omega}(a_{\nu},\mathcal C_{\nu}) \leq M$ and Condition~(\ref{eq-abs}) is equivalent to the existence, for every $\nu \geq 1$, of a point $a'_{\nu} \in \mathcal C_{\nu}$ such that
\begin{equation}\label{eq-abs2}
\lim_{\nu \rightarrow \infty}K_{\Omega}(a'_{\nu},[y_{\nu},z_{\nu}]_{\Omega} \cup [z_{\nu},x_{\nu}]_{\Omega}) = +\infty.
\end{equation}

\vspace{1mm}
- If $[x_{\nu},z_{\nu}]_{\Omega}$ and $[y_{\nu},z_{\nu}]_{\Omega}$ are contained in $V_{y_{\infty}}$, then we get a contradiction since they are both $(A,0)$-quasi-geodesic segments for $K_{\Omega \cap U_{y_{\infty}}}$ and $(\Omega \cap U_{y_{\infty}},K_{\Omega \cap U_{y_{\infty}}})$ is Gromov hyperbolic, so that Remark~\ref{Rem:converse-inf} implies
$$
\sup_{\nu >> 1} K_{\Omega \cap U_{y_{\infty}}}(a'_{\nu},[x_{\nu},z_{\nu}]_{\Omega} \cup [y_{\nu},z_{\nu}]_{\Omega}) < +\infty.
$$

\vspace{1mm}
- If either $[x_{\nu},z_{\nu}]_{\Omega}$ or $[y_{\nu},z_{\nu}]_{\Omega}$ is not contained in $V_{y_{\infty}}$, we may replace it by a $(A,B)$-quasi-geodesic segment for $K_{\Omega \cap U_{y_{\infty}}}$, using again Lemma~\ref{lem-loc}. We construct in that manner a $(A,B)$-quasi-geodesic triangle for $K_{\Omega \cap U_{y_{\infty}}}$ with the three edges  being $\mathcal C_{\nu}$, a quasi-geodesic $\mathcal C'_{\nu}$ joining $y_{\nu}$ to $z_{\nu}$ and a quasi-geodesic $\mathcal C''_{\nu}$ joining $z_{\nu}$ to $x_{\nu}$, such that
$$
[y_{\nu},z_{\nu}]_{\Omega} \subset \mathcal N^{K_{\Omega}}_M\left(\mathcal C'_{\nu}\right)
$$
and
$$
[x_{\nu},z_{\nu}]_{\Omega} \subset \mathcal N^{K_{\Omega}}_M\left(\mathcal C''_{\nu}\right).
$$

Hence, since $(\Omega \cap U_{y_{\infty}},K_{\Omega \cap U_{y_{\infty}}})$ is Gromov hyperbolic, there exists $N>0$ such that
$$
K_{\Omega}(a'_{\nu},\mathcal C'_{\nu} \cup \mathcal C''_{\nu}) \leq K_{\Omega \cap U_{y_{\infty}}}(a'_{\nu},\mathcal C'_{\nu} \cup \mathcal C''_{\nu}) \leq N,
$$
contradicting Condition~(\ref{eq-abs2}).

\vspace{1mm}
\noindent{$\bullet$} Assume, up to extraction, that for every $\nu \geq 1$, $[x_{\nu},y_{\nu}]_{\Omega} \subset V_{y_{\infty}}$ and that either $[x_{\nu},z_{\nu}]_{\Omega}$ or $[y_{\nu},z_{\nu}]_{\Omega}$ is not contained in $V_{y_{\infty}}$. We reproduce the same argument as in the previous case, considering directly $a_{\nu}$ instead of $a'_{\nu}$. We obtain the same contradiction.

\vspace{2mm}
\noindent{\bf Case V. $x_{\infty} = y_{\infty},\ z_{\infty} \neq x_{\infty}$.}

Shrinking $V_{y_{\infty}}$ if necessary, we can assume that $z_{\infty} \not\in V_{y_{\infty}}$. We can also assume as in Case IV, replacing $a_{\nu}$ with $a'_{\nu}$ if necessary, that $[x_{\nu},y_{\nu}]_{\Omega} \subset V_{y_{\infty}}$, for every $\nu \geq 1$. By Lemma~\ref{lem-loc0}, $[x_{\nu},y_{\nu}]_{\Omega}$ is a $(A,0)$-quasi-geodesic segment for $K_{\Omega \cap U_{y_{\infty}}}$.

For every $\nu \geq 1$, let $x'_{\nu} \in [x_{\nu},z_{\nu}]_{\Omega}$ (resp. $y'_{\nu} \in [y_{\nu},z_{\nu}]_{\Omega}$) be such that $[x_{\nu},x'_{\nu}]_{\Omega} \subset V_{y_{\infty}}$ (resp. $[y_{\nu},y'_{\nu}]_{\Omega} \subset V_{y_{\infty}}$) and $d_{Eucl}(x'_{\nu},\partial \Omega)  \geq C$ (resp. $d_{Eucl}(y'_{\nu},\partial \Omega)  \geq C$).
Let
$$
\mathcal C_{\nu}:=[x_{\nu},x'_{\nu}]_{\Omega} \cup [x'_{\nu},y_{\nu}]_{\Omega \cap U_{y_{\infty}}} \cup [y_{\nu},x_{\nu}]_{\Omega}.
$$
By Lemma~\ref{lem-loc0}, the curve $\mathcal C_{\nu}$ is a $(A,0)$-quasi-geodesic triangle for $K_{\Omega \cap U_{y_{\infty}}}$.

By Remark~\ref{Rem:converse-inf},
$$
\sup_{\nu \geq 1}K_{\Omega}(a_{\nu},[x'_{\nu},y_{\nu}]_{\Omega \cap U_{y_{\infty}}} \cup [y_{\nu},x_{\nu}]_{\Omega})
$$
$$
\leq \sup_{\nu \geq 1}K_{\Omega \cap U_{y_{\infty}}}(a_{\nu},[x'_{\nu},y_{\nu}]_{\Omega \cap U_{y_{\infty}}} \cup [y_{\nu},x_{\nu}]_{\Omega}) < +\infty.
$$
From Condition~(\ref{eq-abs}), we obtain
$$
\sup_{\nu \geq 1}K_{\Omega}(a_{\nu},[x'_{\nu},y_{\nu}]_{\Omega \cap U_{y_{\infty}}}) < +\infty.
$$
Let $b_{\nu} \in [x'_{\nu},y_{\nu}]_{\Omega \cap U_{y_{\infty}}}$ be such that $\sup_{\nu \geq 1}K_{\Omega}(a_{\nu},b_{\nu}) < +\infty$.
Since the curve
$$
\mathcal C'_{\nu}:=[y_{\nu},y'_{\nu}]_{\Omega} \cup [y'_{\nu},x'_{\nu}]_{\Omega \cap U_{y_{\infty}}} \cup [x'_{\nu},y_{\nu}]_{\Omega \cap U_{y_{\infty}}}
$$
is a $(A,0)$-quasi-geodesic triangle for $K_{\Omega \cap U_{y_{\infty}}}$, we obtain as above
$$
\sup_{\nu \geq 1}K_{\Omega}(b_{\nu},[y_{\nu},y'_{\nu}]_{\Omega} \cup [y'_{\nu},x'_{\nu}]_{\Omega \cap U_{y_{\infty}}}) < +\infty.
$$
From Condition~(\ref{eq-abs}) we obtain
$$
\sup_{\nu \geq 1}K_{\Omega}(b_{\nu},[y'_{\nu},x'_{\nu}]_{\Omega \cap U_{y_{\infty}}}) < +\infty.
$$
The set $\cup_{\nu \geq 1}[y'_{\nu},x'_{\nu}]_{\Omega \cap U_{y_{\infty}}}$ being relatively compact, we have
$$
\sup_{\nu \geq 1,\ p \in [y'_{\nu},x'_{\nu}]_{\Omega \cap U_{y_{\infty}}}}\left(K_{\Omega}(x'_{\nu},p)\right) < +\infty.
$$
This implies
$$
\sup_{\nu \geq 1}K_{\Omega}(a_{\nu},x'_{\nu}) < +\infty,
$$
which is a contradiction.
\end{proof}

\subsection{Proofs of Theorem~\ref{thm-locglob2} and Corollary~\ref{Cor:localmodel-global}}

The Gromov compactification of a Gromov hyperbolic space has no ``geodesic loops'' and is ``visible'' (in the sense of this paper, but considering the Gromov boundary instead of the Euclidean one)---this can be extrapolated from the proof of \cite[Prop. 2.1]{Cor-Del-Pap}---but, for  completeness, here we prove directly the following lemma in our setting.

\begin{lemma}\label{Lem:model-to-vis}
Let $D\subset\subset \C^d$ be a Gromov model domain. Then $D$ is visible and $\overline{D}$ has no geodesic loops.
\end{lemma}

\begin{proof}
Denote by  $\mathcal I: \overline{D}\to \overline{D}^G$ the inverse of the homeomorphism induced by the identity map of~$D$.

Assume $D$ has a geodesic loop $\gamma:\R\to D$, that is, for every $s,t\in\R$, $K_D(\gamma(s), \gamma(t))=|t-s|$ and the cluster set of $\gamma(t)$ at $t=+\infty$ equals the cluster set of $\gamma(t)$ at $t=-\infty$.

Consider the two geodesic rays $\gamma^+:[0,+\infty)\to D$, $\gamma^+(t):=\gamma(t)$ and $\gamma^-:[0,+\infty)\to D$, $\gamma^-(t):=\gamma(-t)$. Hence, in the Gromov topology, $\lim_{t\to+\infty}\gamma^+(t)=[\gamma^+]$ and  $\lim_{t\to+\infty}\gamma^-(t)=[\gamma^-]$. Since $\mathcal I$ is a homeomorphism, and $\gamma$ is a geodesic loop, there exist $p\in \partial D$ such that $[\gamma^+]=[\gamma^-]= \mathcal I(p)$.

By definition of Gromov topology, this implies that there exists a constant $C>0$ such that $K_D(\gamma^+(t), \gamma^-(t))\leq C$ for all $t\in [0,+\infty)$. Hence, for all $t\geq 0$,
\[
2t=K_D(\gamma^-(t), \gamma^+(t))\leq C,
\]
a contradiction.

Next, let $p, q\in\partial D$, $p\neq q$. We need to show that $p,q$ satisfy the visibility condition with respect to $K_D$. 

We claim that there exist a compact set $K\subset\subset D$, an open neighborhood $V_p$ of $p$ and an an open neighborhood $V_q$ of $q$, such that for every $x\in D\cap V_p$ and every $y\in D\cap V_q$ we have $[x,y]_D\cap K\neq\emptyset$---hence $p,q$ satisfy the visibility condition.

We argue by contradiction and assume the claim is false. Then there exist a sequence $\{z_n\}\subset D$ converging to $p$ and a sequence $\{w_n\}\subset D$ converging to $q$ such that for every compact subset $K\subset D$, we have $[z_n,w_n]_D\cap K=\emptyset$ for every $n$ sufficiently large. 

Since $K_D$ is complete, up to extracting subsequences, we can assume that $\{T_n:=K_D(z_0,z_n)\}$ and $\{R_n:=K_D(z_0,w_n)\}$ are strictly increasing. 

For all $n$, let $\gamma_n:[0,T_n]\to D$ be a geodesic such that $\gamma_n(0)=z_0$, $\gamma_n(T_n)=z_n$  and let $\eta_n:[0,R_n]\to D$ be a geodesic such that $\eta_n(0)=z_0$, $\eta_n(R_n)=w_n$. Since $\{\gamma_n\}$ (and $\{\eta_n\}$) are equicontinuous and equibounded on compacta, up to subsequences, we can assume that $\{\gamma_n\}$ converges uniformly on compacta to a geodesic ray $\gamma:[0,+\infty)\to D$ and $\{\eta_n\}$ converges uniformly on compacta to a geodesic ray $\eta:[0,+\infty)\to D$. Taking into account that $\mathcal I$ is a homeomorphism, it follows that $\{z_n\}$ converges in the Gromov topology of $D$ to $\mathcal I(p)=[\gamma]\in\partial_G D$ and $\{w_n\}$ converges in the Gromov topology of $D$ to  $\mathcal I(q)=[\eta]\in\partial_G D$, where $\partial_G D$ is the Gromov boundary of $D$ and $[\gamma], [\eta]$ are the points in $\partial_G D$ represented by $\gamma$ and $\eta$. 

Let $\delta>0$ be the Gromov constant of $D$ for which every geodesic triangle in $D$ is $\delta$-thin. Since $p\neq q$, hence $[\gamma]\neq [\eta]$, it follows from the definition of Gromov topology that there exists $s_0>0$ such that $K_D(\gamma(t), \eta(s_0))\geq 2\delta$ for all $t\geq 0$. Taking into account that $\{\gamma_n\}$ and $\{\eta_n\}$ converge uniformly on compacta to $\gamma$ and $\eta$ respectively,  there exists $N \in \N$ such that $T_n> s_0+\delta$, $R_n>s_0$ for all $n\geq N$, and
\[
K_D(\gamma_n(t), \eta_n(s_0))>\delta, \quad t\in [0,s_0+\delta], \quad n\geq N.
\]
Note that, for $T_n\geq t>s_0+\delta$,
\[
K_D(\gamma_n(t), \eta_n(s_0))\geq K_D(\gamma_n(t),\gamma_n(0))-K_D(\eta_n(s_0),\eta_n(0))=t-s_0>\delta.
\]
Therefore,
\[
K_D(\gamma_n(t), \eta_n(s_0))>\delta \quad n\geq N,\  t\in [0, T_n].
\]

Since geodesic triangles are $\delta$-thin, it follows that for every $n\geq N$  there exists a point  $\xi_n\in [z_n,w_n]_D$ such that $K_D(\eta_n(s_0), \xi_n)\leq \delta$. Since $\{\eta_n(s_0)\}$ converges to $\eta(s_0)$, this implies that $[z_n,w_n]_D$ intersects a compact set $K\subset\subset D$ for every $n>N$, a contradiction.
\end{proof}

\begin{proof}[Proof of Theorem~\ref{thm-locglob2}]
According to  \cite[Theorem~3.3]{BNT} (see also Lemma~\ref{Lem:model-to-vis}) if $\Omega$ is a Gromov model domain then $\Omega$ is   Gromov hyperbolic, visible and has no  geodesic loops. Conversely, by  Theorem~\ref{thm-locglob}, we have only to show that $\overline{\Omega}$ has no geodesic loops.

Let $U_p$ be the open neighborhood of $p$ such that $U_p\cap\Omega$ is connected, $(U_p\cap \Omega, K_{U_p\cap\Omega})$ is complete hyperbolic and Gromov hyperbolic. Arguing as in the proof of Lemma~\ref{Lem:equiv-vis}, using the hypothesis that $\Omega$ has no local geodesic loops at $p$, we see that also $\overline{U_p\cap \Omega}$ has no geodesic loops at $p$.

Now, arguing by contradiction, assume  $\gamma:(-\infty,+\infty)\to \Omega$ is a geodesic loop in $\overline{\Omega}$. Since $\Omega$ has the visibility property, by \cite[Lemma~3.1]{BNT}, there exists a point $p\in \partial \Omega$ so that $\lim_{t\to\pm\infty}\gamma(t)=p$.

According to Lemma~\ref{lem-loc}, it follows that $\overline{U_p\cap \Omega}$ has a quasi-geodesic loop with vertex $p$. By Remark~\ref{Rem:converse-inf}, it follows that $\overline{U_p\cap \Omega}$ has a geodesic loop with vertex $p$, contradiction.
\end{proof}

\begin{proof}[Proof of Corollary~\ref{Cor:localmodel-global}]
Since $U_p\cap \Omega$ is a Gromov model domain for every $p$, then, by Lemma~\ref{Lem:model-to-vis}, $U_p\cap \Omega$ is visible and $\overline{U_p\cap \Omega}$ has no geodesic loops. Hence $\Omega$ is a  locally Gromov hyperbolic, locally visible, bounded domain and has no local geodesic loops, hence it is a Gromov model domain by Theorem~\ref{thm-locglob2}.
\end{proof}

\section{Bounded Gromov hyperbolic  $\C$-convex domains with Lipschitz boundary}
\label{sec:c-convex}

In this section we prove (see Proposition~\ref{cc}) that if $\Omega$ is a bounded $\C$-convex domain with Lipschitz boundary and $(\Omega, K_\Omega)$ is Gromov hyperbolic  then $\Omega$ is a Gromov model domain. The proof uses an extension to  Lipschitz boundary 
of a result proven by Zimmer in the $C^1$-smooth case  \cite[Theorem 1.4]{Z2}, which allows to construct suitable quasi-geodesics and prove visibility. Then, an argument similar to the one in \cite{BGZ} allows to show that there are no geodesic loops.

In order to prove the result, we need some preliminary. 

Let $\Omega\subset \C^d$ be a bounded Gromov hyperbolic $\C$-convex domain with Lipschitz boundary. We  can cover $\partial \Omega$ by a finite collection of open sets $U_j$ such that for each $j$
there is a set of affine coordinates, obtained by choosing a base point in $\partial \Omega$
and an orthonormal basis, in which
$$
U_j= \{(z_1, z_2, \dots, z_d): |\Re z_1 | < r_j, (\Im z_1)^2+|z_2|^2 +\cdots+|z_d|^2 <R_j^2\},
$$
for some $R_j, r_j>0$, and $\Omega \cap U_j= \{ z\in U_j: \Re z_1 < F_j(\Im z_1, z_2, \dots, z_d)\}$,
where $F_j$ is a Lipschitz function. Let $V_j$  be the vector $(1,0,\dots,0)$ in the coordinates corresponding to $U_j$.

\begin{lemma}\label{Lem:G} There exist $A>1, B\geq 0$, a subcovering by $U'_j\subset \subset U_j$ and $\epsilon_j\in(0,r_j)$
so that for any $p\in U'_j\cap\partial \Omega$,
$t\mapsto p- \epsilon_j e^{-t} V_j$, $t\ge 0$, is a $(A,B)$-quasi-geodesic.
\end{lemma}

\begin{proof} The Lipschitz condition on the boundary of $\partial \Omega$ implies that if $z\in U'_j$, $z$ close enough to
$\partial \Omega$, and $z=p- \epsilon_j e^{-t} V_j=:\gamma(t)$ with $p\in U'_j\cap\partial \Omega$, then
there exists $C>0$ such that the Euclidean ball $B(\gamma(t), C\epsilon_j e^{-t})\subset \Omega$.  From this we deduce
that there are constants $a,b>0$ such that
\[
K_\Omega (\gamma(t), \gamma(t+a))
\le K_{B(\gamma(t), C\epsilon_j e^{-t})} (\gamma(t), \gamma(t+a)) \le b,
\]
 for any $t$. Thus, there exist $A> 1, B\geq 0$ so that for any $t, t'$, 
\[ 
 K_\Omega (\gamma(t), \gamma(t')) \le A|t-t'|+B.
 \]

On the other hand, since $\Omega$ is a $\C$-convex domain,
by \cite[Lemma 3.3]{Z2}, for $t\ge t'$,
\[
K_\Omega(\gamma(t), \gamma(t')) \ge  \frac14
\log\frac{\|\gamma(t)-p\|}{\|\gamma(t')-p\|} = \frac14 |t-t'|,
\]
so we have the
reverse inequality. The claim is proven.
\end{proof}

\begin{proposition}
\label{nodisc}
If $\Omega\subset \C^d$ is a bounded Gromov hyperbolic $\C$-convex
domain with Lipschitz boundary, then $\partial \Omega$ cannot contain a non-trivial
affine complex disc.
\end{proposition}

\begin{proof}
We follow the arguments of the proof of \cite[Theorem 1.4]{Z2}.

Let $\varphi:\C\longrightarrow \C^d$  be a complex affine map so that
$\Delta := \varphi (\C) \cap \partial \Omega$ has non-empty relative interior  in $\varphi (\C)$.

Take $p$ a point of the relative boundary of $\Delta$ in $\partial \Omega$. Let $U'_0$ be an open neighborhood of $p$ with coordinate system given by Lemma~\ref{Lem:G}. Let $\Delta_0:=\Delta \cap U'_0$. For every $q\in \Delta_0$ we define
\[
q_t:=q-\epsilon_0 e^{-t} V_0, \quad t\geq 0.
\]
Consider any two points $p',p''\in \Delta_0$.
Because of Lemma~\ref{Lem:G}, the curves given by $p'_t$, $p''_t$, $t\geq 0$, are quasi-geodesics.
We can choose $\Delta_1 \subset \subset \Delta_0$ a connected, simply connected relative open set so that $p',p''\in \Delta_1$.
Then $p'_t,p''_t \in \Delta_1 - \epsilon_0 e^{-t} V_0 \subset \Omega$ for every $t\geq 0$,
and letting $\varphi_t(\zeta):=\varphi(\zeta)- \epsilon_0 e^{-t} V_0$, we see that the preimages of $p'_t,p''_t $
under $\varphi_t$ are contained in $\varphi^{-1}(\Delta_1)$, a fixed relatively compact subset of $\varphi^{-1}(\Delta_0)$.
It follows that 
\[
\sup_{t\geq 0} K_\Omega(p'_t,p''_t )\le C
\]
for some $C>0$ depending on $\Delta_1$.  Hence, by  \cite[Proposition 4.3]{Z2}, there exists $M>0$ (related to the constant involved in the definition of Gromov
hyperbolicity of $\Omega$)  such that 
\begin{equation}\label{Eq-estimZim}
\sup_{t \geq 0}K_\Omega(p'_t,p''_t )\le M+K_\Omega(p'_0, p''_0).
\end{equation}

Let $K \subset \Omega$ be the compact subset defined by
\begin{equation*}\label{Eq:cmpL}
K:=   \left\{z \in \overline U'_0: \Re z_1 -F_0(\Im z_1, z_2, \dots, z_d) \le -\epsilon_0 \right\}.
\end{equation*}

Let
\[
N:=\max_{x,y\in K} K_\Omega(x,y).
\]
Hence, it follows by \eqref{Eq-estimZim} that for every $\sigma,\tau\in\Delta_0$,
\[
\sup_{t \geq 0}K_\Omega(\sigma_t,\tau_t )\leq M+N.
\]
Now, let $q\in \Delta_0$. Therefore, if $\{p^n\}\subset \Delta_0$ is a sequence converging to $p$,   and  taking into account that $\lim_{n\to+\infty}p^n_t=p_t$ for every $t\geq 0$, we have 
\[
\sup_{t\geq 0} K_\Omega(q_t,p_t )\leq M+N.
\]
By \cite[Proposition 3.5]{Z2}
it follows then that $p, q$ are contained in the relative interior of an affine disc in $\partial \Omega$: a contradiction to the definition of $p$.
\end{proof}

Note that Proposition~\ref{nodisc} recovers the convex case since any convex domain has Lipschitz boundary.

The previous proof shows in fact that if $D$ is a bounded
Gromov hyperbolic $\C$-convex domain and $L$ is a complex line such that the relative
interior $\Delta_L$
of $L\cap\partial D$ in $L$ is nonempty, then $\partial D$ is not Lipschitz near any boundary
point of a connected component of $\Delta_L$. Such domains exist, as shown by the next example:

\begin{example} 
The example in \cite[Prop. 1.9]{Z2}
is a $\C$-convex domain $\Omega$
which is Gromov hyperbolic, contains many complex affine discs in its boundary,
which is Lipschitz (or more regular) except at the relative boundaries of the connected components
of $L\cap\partial \Omega$ when $L$ is a complex line.

Let
\[
\mathcal C_2 := \left \{(w_0, w) \in \C \times \C^d : \Im w_0 > \|w\|\right\},
\]
where $\|\cdot\|$ stands for the Euclidean norm in $\C^d$.  The domain $\mathcal C_2$  is unbounded,
convex and Gromov hyperbolic.

Let
$f(w_0,w):= \left( \frac1{i+w_0}, \frac{w}{i+w_0} \right)$. Note that $f$ is
a biholomorphism which preserves complex lines, so $\Omega:= f(\Co_2)$ remains
Gromov hyperbolic and $\C$-convex.

Then let
\[
\Omega := \left \{(z_0, z) \in \C \times \C^d : \rho(z_0, z):= \Im z_0+ |z_0|^2 + |z_0|\|z\| <0 \right\}.
\]
Note that $\Omega$ can also be seen as a (bounded) Hartogs domain over the disc $D(-\frac{i}2, \frac12)$ of center $-\frac{i}2$ and radius $\frac12$,
in the $z_0$-plane, given by
\[
\|z\| < \Phi (z_0) := \frac{\Im (- z_0)- |z_0|^2}{|z_0|}.
\]
The cluster set of $\Phi$ at $0$ is $[0,1]$, so $\partial \Omega \cap  \{z_0=0\}
= \{0\}\times \overline {\mathbb B}^d$. In fact, $\{0\}\times  {\mathbb B}^d$
is the union of all the open analytic discs contained in the boundary, since
the complex Hessian of $\rho$ is positive definite at all the other points (when
it is defined).

Near any point  $(p_0,p)\in\partial \Omega$ with $|p_0|\|p\|\neq 0$, and at $(0,0)$,
$\rho$ is differentiable and $\nabla \rho(p_0,p)\neq 0$, so $\partial \Omega$
is smooth (actually $\mathcal C^\infty$).

Near points $(p_0,0) \in\partial \Omega$ with $p_0\neq 0$, $\Phi$ is
smooth with non-vanishing derivative, so the boundary is Lipschitz, and not $\Co^1$.

For any $R\in (0,1)$, we have $\{r\eit \}\times D(0,R) \subset \Omega$
if and only if $r< \cos(\theta +\frac\pi2) -R$.  This defines a region
$P_R$ bounded by a Lipschitz graph
near $0$ (tangent to a cone of aperture $2 \arccos R$ centered on the negative
imaginary axis).  So at any boundary point $(0, p)$ with $\|p\|=R\in (0,1)$, $\partial \Omega$
is Lipschitz and not $\Co^1$. Note also that $P_R$ is contained in a disc
of radius $1-R$ around $0$.

Finally, for $(0,p)$ with $\|p\|=1$, let us consider points of the
form $z_t:=(0,p)+t(V_0,V)$, where $t>0$ and $(V_0,V)\in \C \times \C^d$.
If the boundary was Lipschitz near this
$(0,p)$, there should be a whole open cone of vectors $(V_0,V)$
such that $\rho(z_t)<0$ for $0<t<\epsilon$, with $\epsilon$ depending on the vector.
However, $\lim_{t\to0^+} t^{-1} \rho(z_t) = \Im V_0 + |V_0|$, so we must have
$V_0\in i\R_-$, which is a condition with empty interior.

A pair of boundary points $\{p,q\}$ is not visible if and only if $p,q\in\{0\}\times\overline{\mathbb B}^d$ (and so
the conclusion of Proposition
\ref{cc} does not hold in this example by \cite[Theorem 3.3]{BNT}).

Indeed, if  $p$  belongs to $\partial \Omega \setminus \{z_0=0\}$,
then we can choose $W$ a neighborhood of $p$ such that $(\Omega\cap W)\cap \{z_0=0\} =\emptyset$,
so that  $\Omega\cap W$ is Gromov hyperbolic and has Lipschitz boundary, so it has
the visibility property by the next Proposition \ref{cc}. Then it is easy to show that no
geodesics from $p$ to $q$ can escape from all compacta of $\Omega$.  The same argument
holds when $q\notin \{z_0=0\}$.

If $p,q\in\{0\}\times\overline{\mathbb B}^d$, we can approach them by sequences $\{p_\nu\}, \{q_\nu\}$
that approach $\{z_0=0\}$ much faster than they approach
$\partial \Omega \setminus (\{0\}\times{\mathbb B}^d)$, and construct curves inside
analytic discs parallel to $\{z_0=0\}$ which give a shorter Kobayashi length than any
curve passing through a compactum inside $\Omega$.
\end{example}

Now we are ready to state and prove the main result of this section:

\begin{proposition}\label{cc} If $\Omega\subset \C^d$ is a bounded Gromov hyperbolic $\C$-convex
domain with Lipschitz boundary, then $\Omega$ is a Gromov model domain.
\end{proposition}
\begin{proof}Note that $(\Omega, K_\Omega)$ is complete hyperbolic (see {\sl e.g.} \cite[Proposition 3]{NPZ}).

First we show that $\Omega$ is visible. Using coordinates as in Lemma~\ref{Lem:G}, we define  a compact set $L \subset \Omega$ by
\begin{equation*}
L:=  \left( \Omega \setminus \bigcup_j U_j \right) \cup \bigcup_j \left\{z \in \overline U'_j: \Re z_1 -F_j(\Im z_1, z_2, \dots, z_d) \le -\epsilon_j \right\},
\end{equation*}
and a (one-sided) neighborhood $\mathcal V$ of $\partial \Omega$ in $\overline \Omega$
by $\mathcal V:= \Omega \setminus L$. For any $z\in\mathcal V$, we can choose
a point
\[
z' = (F_j(\Im z_1, z_2, \dots, z_d)-\epsilon_j+i \Im z_1, z_2, \dots, z_d) \in L,
\]
when $z\in U'_j$ (if $z$ belongs to several
open sets $U'_j$, just pick one, the choice does not need to be continuous).

Assume by contradiction that $\Omega$ is not visible. Then we can find two sequences $\{p_n\}, \{q_n\}\subset \Omega$, converging, respectively, to $p\in\partial\Omega$ and $q\partial\Omega$, $p\neq q$ such that $[p_n,q_n]_\Omega$ eventually escapes any compact subset of $\Omega$. 

Up to subsequences, we can find a sequence $\{r_n\}$ such that $r_n\in [p_n,q_n]_\Omega$ for all $n$, and $\{r_n\}$ converges to a point $r\in\partial\Omega\setminus\{p,q\}$. 

For every $a, b\in \C^d$, denote by $[a,b]$ the real Euclidean line segment from $a$ to $b$. By  Lemma~\ref{Lem:G}, the Geodesic stability theorem and \cite[Observation 4.4]{Z2}, there exists $\delta'>0$ such that for any $n>n_0$, the   $(A,B)$-quasi-geodesic rectangle
$L_n=[p_n,q_n]_\Omega \cup[q_n,q_n']\cup[q_n',p_n']_\Omega\cup[p_n',p_n]$ is
$\delta'$-thin (namely, each side is contained in the $\delta'$-tubular neighborhood of the union of the other sides).

Since $K_\Omega$ is complete, $\cup_n [q_n',p_n']_\Omega$ is relatively compact in $\Omega$ and $K_\Omega(r_n,L_n\setminus[p_n,q_n]_\Omega)\leq\delta'$, there exists a compactly divergent sequence $\{s_n\}\subset [q_n,q_n']\cup[p_n',p_n]$ such that for all $n$,
 \[
 K_\Omega(r_n,s_n)\leq\delta'.
 \] 
 Therefore, by construction, $\{s_n\}$ accumulates to either $p$ or $q$ (or both). Say, $\lim_{n\to \infty}s_n=p$.

But then, \cite[Lemma A.2]{BG0} (or \cite[Proposition 3.5]{Z2}, which uses $\C$-convexity but no
additional smoothness), implies that $\partial \Omega$ contains affine discs
through $r$ and $p$, contradicting  Proposition~\ref{nodisc}. Therefore $\Omega$ is visible.

Finally, we show that $\Omega$ has no geodesic loops, so that $\Omega$ is a Gromov model domain by \cite[Thm. 3.3]{BNT}. 

By contradiction, let $\gamma:(-\infty,+\infty)\to\Omega$ be a geodesic loop. Since $\Omega$ is visible, by \cite[Lemma 3.1]{BNT}, every geodesic ray lands at a boundary point, hence, there exists $p\in \partial\Omega$ such that $\lim_{t\to\pm \infty}\gamma(t)=p$. 

Let $\gamma^+:[0,+\infty)\to \Omega$ be defined by $\gamma^+(t):=\gamma(t)$ and $\gamma^-:[0,+\infty)\to \Omega$ be defined by $\gamma^-(t):=\gamma(-t)$. 

We are going to show that there exists $C>0$ such that for all $T\geq 0$,
\begin{equation}\label{Eq:bounded-loop-C-convex}
K_\Omega(\gamma^+(T), \gamma^-(T))\leq C.
\end{equation}
Assuming this for the moment, we see that for all $T\geq 0$,
\[
2T=K_\Omega(\gamma^+(T), \gamma^-(T))\leq C,
\]
and thus a contradiction. 

The argument to prove \eqref{Eq:bounded-loop-C-convex} is similar to the one in the proof of \cite[Lemma 6.5]{BGZ}, replacing the quasi-geodesics defined by real segments for the convex case, with the quasi-geodesics defined in Lemma~\ref{Lem:G}. For the sake of clarity, we sketch it here.

Let $U_j$ be a neighborhood of $p$  defined  as before $U_j'$ be as in Lemma~\ref{Lem:G} so that $p\in U_j'$.   

For $t$ sufficiently large, since $\gamma^{+}(t)$ converges to $p$, we can define $q^+(t)\in \partial\Omega\cap U'_j$ in such a way that there exists $s\geq 0$ so that $\gamma^+(t)=q(t)^+-\epsilon_je^{-s}V_j$. Similarly define $q^-(t)$. Also, define $Q_t^{\pm}:[0,+\infty)\to \Omega\cap U_j$  as $Q^{\pm}_t(r)=q^{\pm}(t)-\epsilon_je^{-r}V_j$. By Lemma~\ref{Lem:G},  $Q^{\pm}_t$ are $(A,B)$-quasi-geodesics for $t$ sufficiently large so that $q^{\pm}(t)\in U_j'\cap\Omega$.

Now, fix $T>0$ sufficiently large so that for $t\geq T$, $Q_t^\pm$ are $(A,B)$-quasi-geodesics (here, with some abuse of notation we use $Q_t^\pm$ for denoting also the image of the curve $Q_t^{\pm}$). Since $Q_t^+(0)\in L$ ($L$ the compact subset of $\Omega$ defined as above), the curve $Q_t^+\cup [Q_t^+(0), \gamma(0)]_\Omega$ is a $(A,B')$-quasi-geodesic for some $B'\geq B$. 
Hence, by the Geodesic Stability Theorem,  there exists some $R>0$ such that for every $t>T$ there exists $z_t\in Q_t^+$  such that 
\[
K_\Omega(z_t, \gamma^+(T))\leq R.
\]
Note that $\{z_t\}$ is relatively compact in $\Omega$. Since $q^{\pm}(t)\to p$ as $t\to+\infty$, and $Q^{\pm}_t$ are real segments parallel to $V_j$, it follows that for every $t>T$ there exists $w_t\in Q^-_t$ such that 
\[
K_\Omega(z_t,w_t)\leq R.
\]
Again by the Geodesic Stability Theorem, we can find $s\geq 0$ such that 
\[
K_\Omega(\gamma^-(s), w_t)\leq R.
\]
Then, $K_\Omega(\gamma^+(T), \gamma^-(s))\leq 3R$. Therefore,
\begin{equation*}
\begin{split}
K_\Omega(\gamma^+(T), \gamma^-(T))&\leq K_\Omega(\gamma^+(T), \gamma^-(s))+K_\Omega(\gamma^-(s), \gamma^-(T))\\&=K_\Omega(\gamma^+(T), \gamma^-(s))+|T-s|\\&=K_\Omega(\gamma^+(T), \gamma^-(s))+|K_\Omega(\gamma^+(T), \gamma(0))-K_\Omega(\gamma^-(s), \gamma(0))|\\
&\leq 2 K_\Omega(\gamma^+(T), \gamma^-(s))\leq 3R.
\end{split}
\end{equation*}
Thus, \eqref{Eq:bounded-loop-C-convex} follows, and we are done.
\end{proof}

\end{document}